\documentclass[UTF8,reqno,11pt]{amsart}
\usepackage{framed}
\usepackage{color,xcolor}
\usepackage{amsmath,amssymb,amsthm,amsfonts,dsfont,mathrsfs,bbm}
\usepackage{geometry}
\usepackage{fancyhdr}%页眉页脚格式
\usepackage{marginnote}
\usepackage[bookmarks=true,colorlinks,linkcolor=red]{hyperref}
\allowdisplaybreaks
\usepackage{graphicx}
\usepackage{tikz}
\usetikzlibrary{positioning}
\usepackage{stmaryrd}
\usepackage{esint}
\usepackage{functan,mathtools,extarrows}
\numberwithin{equation}{section}\newcounter{MyCounter}[section]

\newtheorem{lemma}[MyCounter]{Lemma}
\newtheorem{theorem}[MyCounter]{Theorem}
\newtheorem{nota}[MyCounter]{Notation}

\newtheorem{corollary}[MyCounter]{Corollary}
\newtheorem{remark}[MyCounter]{Remark}
\newcommand{\ol}[1]{\overline{#1}}

\def\vp{\varphi}
\def\TP{\overline{\partial}}

\def\d{\mathrm{d}}
\def\DD{\mathcal{D}}
\def\ddt{\frac{\d}{\d t}}

\def\A{\mathcal{A}}
\def\b{\mathrm{btm}}

\def\H{\mathcal{H}}
\def\i{\mathrm{in}}
\def\K{\mathcal{K}}
\newcommand{\KM}{\widetilde{K}}
\def\N{\mathbf{N}}
\def\p{\partial}
\def\ppk{\p^\vp}
\def\Q{\mathbf{Q}}
\def\R{\mathbb{R}}

\def\t{\mathrm{top}}
\def\T{\mathbb{T}}
\def\V{\mathbf{V}}

\newcommand{\vbi}{\int_0^t |\vb(\tau)|_{\dot{W}^{1,\infty}}\,d\tau}
\newcommand{\pp}{\p^{\varphi}}
\newcommand{\Dtp}{D_t^{\varphi}}
\newcommand{\nab}{\nabla}
\newcommand{\vb}{\overline{v}}
\newcommand{\omp}{\omega^{\varphi}}
\newcommand{\tu}{\tilde{u}}
\newcommand{\blue}{}

\title[The Breakdown Criterion of Incompressible Euler Equations with Surface Tension]{A Generalized Beale-Kato-Majda Breakdown Criterion for the free-boundary problem in Euler Equations with Surface Tension}
\author[Chenyun Luo]{Chenyun Luo}\thanks{ (Chenyun Luo)
Department of Mathematics, The Chinese University of Hong Kong, Shatin, NT, Hong Kong. Email: cluo@math.cuhk.edu.hk.
CL is supported in part by the Hong Kong RGC grant CUHK-24304621.}
\author[Kai Zhou]{Kai Zhou}\thanks{ (Kai Zhou) The Department of Mathematics and Institute of Mathematical Sciences, The Chinese University of Hong Kong, Shatin, NT, Hong Kong. Email: kzhou@math.cuhk.edu.hk
}

\date{\today}
\begin{document}
\maketitle
\begin{abstract}
It is shown in Ferrari \cite{Ferrari-1993CMP} that if $[0, T^*)$ is the maximal time interval of existence of a smooth solution of the incompressible Euler equations in a bounded, simply-connected domain in $\R^3$, then $\int_0^{T^*}\|\omega(t,\cdot)\|_{L^\infty}\d t=+\infty$, where $\omega$ is the vorticity of the flow. Ferrari's result generalizes the classical Beale-Kato-Majda \cite{BKM-1984CMP}'s breakdown criterion in the case of a bounded fluid domain.

   In this manuscript, we show a breakdown criterion for a smooth solution of the Euler equations describing the motion of an incompressible fluid in a bounded domain in $\R^3$ with a free surface boundary. The fluid is under the influence of surface tension.  In addition, we show that our breakdown criterion reduces to the one proved by Ferrari \cite{Ferrari-1993CMP} when the free surface boundary is fixed.  Specifically, the additional control norms on the moving boundary will either become trivial or stop showing up if the kinematic boundary condition on the moving boundary reduces to the slip boundary condition.\\
   \noindent\textbf{Keywords.} Breakdown Criterion, Incompressible Euler Equations, Surface Tension, Free-boundary Problem.\\
   \noindent\textbf{2020 Mathematics Subject Classification.} 35Q35, 35R35, 76B03, 76B45
\end{abstract}
\setcounter{tocdepth}{1}
\tableofcontents

%%%%%%%%%%%%%%%%%%%%%%%%%%%%%%%%%%%%%%（正文）%%%%%%%%%%%%%%%%%%%%%%%%%%%%%
\section{Introduction}
We consider the Euler equations modeling the motion of an incompressible fluid in a domain with a moving boundary in $\R^3$:
\begin{equation}\label{eqs-incomp-Euler}
 \begin{cases}
  \p_t u + u\cdot\nabla u + \nabla p = 0,& \quad\i\; \DD_t, \\
  \nabla \cdot u = 0,&  \quad\i\; \DD_t,
  \end{cases}
\end{equation}
where  $u:= u(t,y)$, $p:= p(t,y)$ represent the velocity and pressure of fluid, respectively. Also, for each fixed $t$,
\begin{equation*}
  \DD_t = \left\{(y',y_3)\in \R^3\,\big|\,y':=(y_1,y_2)\in\T^2,\,-b < y_3\leq \psi(t,y')\right\}
\end{equation*}
denotes the moving fluid domain.
The boundary of $\DD_t$ is given by $\p\DD_t = \p\DD_{t,\t}\cup \p\DD_{t,\b}$, where the moving boundary $\p\DD_{t,\t}$ is determined by a graph
$$\p\DD_{t,\t} = \left\{(y',y_3)\in \R^3\,\big|\,y_3 = \psi(t,y')\right\},$$
and
$$\p\DD_{t,\b} = \left\{(y',y_3)\in \R^3\,\big|\,y_3 = -b\right\}$$
is the fixed finite bottom.

The initial and boundary conditions of the system \eqref{eqs-incomp-Euler} are
\begin{equation}\label{cond-in-bdy}
  \begin{aligned}
  \text{(IC)}& \qquad u(0,\cdot):= u_0, \quad \psi(0, \cdot):=\psi_0;\\
  \text{(BC)}& \qquad\left\{\begin{aligned}
  &\p_t\psi = u\cdot N,\quad N:=(-\p_{y_1}\psi, -\p_{y_2}\psi, 1)^T,&\quad \text{on}\,\,\p\DD_{t,\t},\\
  &p = \sigma\H,&\quad \text{on}\,\,\p\DD_{t,\t},\\
  &u\cdot n = 0, \quad n:=(0,0,1)^T,  &\quad \text{on}\,\,\p\DD_{t,\b}. \end{aligned}\right.
  \end{aligned}
\end{equation}
Here, we denote by
$
\H
$ is the mean curvature of the free boundary of the fluid domain, while $\sigma>0$ is the surface tension coefficient.
%Moreover, let $n=(0,0,1)^T$. Then by restricting the momentum equation (the first equation of \eqref{eqs-incomp-Euler}) on $\p\DD_{t,\b}$ and taking the normal component, one can derive from the boundary condition $u\cdot n=0$ on $\p\DD_{t, \b}$ that
%\begin{align}
%\nabla p\cdot n\big|_{\p\DD_{t,\b}}=0. \label{cond-in-fixedbdy}
%\end{align}
Finally, we point out that the local existence theory requires that $\p\DD_{t,\t}\cap \p\DD_{t,\b}=\emptyset$ within the interval of existence $[0, T_0]$. To achieve this, we may set $\|\psi_0\|_{L^\infty(\T^2)}\leq 1$ and $b>10$. Then the continuity of $\psi(t,\cdot)$ guarantees that $\|\psi(t,\cdot)\|_{L^\infty(\T^2)}\leq 10$ holds for all $t\in [0,T_0]$.
\subsection{Fixing the Fluid Domain}
Let
$$
\Omega:= \big\{(x',x_3)\in \R^3\,\big|\,x':=(x_1,x_2)\in \T^2,\;-b < x_3 \leq 0\big\},
$$
with $\p\Omega = \Gamma_{\t}\cup \Gamma_{\b}$, where
$$
\Gamma_{\t} := \{x_3=0\},\quad \Gamma_{\b} := \{x_3=-b\}.
$$
For each fixed $t\geq 0$, we consider a family of mappings $\Phi(t, \cdot): \Omega \rightarrow \DD_t$ given by
\begin{equation}\label{op-Phi}
  \Phi(t,x', x_3)=(x',\varphi(t, x', x_3)),
\end{equation}
with
\begin{equation}\label{eq-phi}
\vp(t,x',x_3) = x_3+\chi(x_3)\psi(t,x').
\end{equation}
Here, $\chi \in C_0^\infty(-b,0]$ a cut-off function verifying
\begin{align}\label{ineq-chi}
\begin{aligned}
  \|\chi'\|_{L^\infty(-b, 0]}\leq \frac{1}{\|\psi_0\|_{L^\infty(\T^2)}+1},\quad& \|\chi''\|_{L^\infty(-b,0]}+\|\chi'''\|_{L^\infty(-b,0]} \leq C,\quad\text{for some generic}\,\,C>0,\\
  &\text{and}\,\,\chi =1 \quad\text{on}\,\,\,\,(-\delta_0, 0],
  \end{aligned}
\end{align}
holds for some $\delta_0>0$ sufficiently small. Note that the first condition in \eqref{ineq-chi} yields that
\begin{equation}
  \p_3\vp(0,x',x_3) = 1+\chi'(x_3)\psi(0,x')\geq 2c_0,
\end{equation}
for some $c_0>0$, and thus we infer from the local existence theory that
\begin{equation} \label{p_3 varphi greater than 0}
\p_3 \vp (t, x', x_3) \geq c_0, \quad \forall t\in [0, T_0],
\end{equation}
which guarantees that $\Phi(t,\cdot)$ is a diffeomophism (see Subsection \ref{sect. 2.1}). It can be seen that $\Gamma_{\t}$ and $\Gamma_{\b}$ respectively correspond to the moving surface boundary $\p\DD_{t,\t}$ and the fixed finite bottom $\p\DD_{t, \b}$ through $\Phi(t, \cdot)$.

We denote respectively by
\begin{equation}\label{eq-vel-press}
  v(t,x):= u(t,\Phi(t,x)),\qquad q(t,x) := p(t,\Phi(t,x)),
\end{equation}
the velocity and pressure defined on the fixed domain $\Omega$.
\begin{nota} [Coordinates and Derivatives] \label{nota deriv} The following notations will be used throughout this manuscript.
\begin{itemize}
\item [i.] We denote by $\p_i:= \frac{\p}{\p x_i}, i=1,2,3$ \blue{the spatial derivatives with respect to the $x$-coordinates.}
\item [ii.] \blue{We denote by $y_i =\Phi_i(t,x)$, $i=1,2,3$ the Eulerian spatial coordinates, and by $\nab_i:= \frac{\p}{\p y_i}$ the Eulerian spatial derivatives.}
\item [iii.] We use $\TP:= (\p_1, \p_2)$ to indicate tangential spatial derivatives.
\end{itemize}
\end{nota}
Then, we see that
\begin{equation}\label{eq-uvpq}
  \nabla_\alpha u \circ \Phi = \p_\alpha^\vp v,\quad \nabla_\alpha p \circ \Phi = \p_\alpha^\vp q,\; \alpha = t,1,2,3.
\end{equation}
where
\begin{equation}\label{op-diff-fix}
  \begin{aligned}
    \p_t^\vp = & \p_t - \frac{\p_t \vp}{\p_3 \vp}\p_3,\\
    \p_a^\vp = & \p_a - \frac{\p_a \vp}{\p_3 \vp}\p_3,\quad a = 1,\,2,\\
    \p_3^\vp = & \frac{1}{\p_3 \vp}\p_3.
  \end{aligned}
\end{equation}

 On the other hand,  since $\mathcal{H} = -\TP\cdot \left(\frac{\TP \psi}{\sqrt{1+|\TP\psi|^2}}\right)$, the boundary condition in \eqref{cond-in-bdy} is turned into
\begin{equation}\label{cond-in-bdy-fix}
  \left\{\begin{aligned}
  &\p_t\psi = v\cdot N,\quad N:=(-\p_1\psi, -\p_2\psi, 1)^T,&\quad \text{on}\,\,\Gamma_{\t},\\
  &q = -\sigma \TP\cdot \left(\frac{\TP \psi}{\sqrt{1+|\TP\psi|^2}}\right),&\quad \text{on}\,\,\Gamma_{\t},\\
  &v\cdot n = 0, &\quad \text{on}\,\,\Gamma_{\b}. \end{aligned}\right.
\end{equation}

Let
$$
D_t^\vp = \p_t^\vp + v\cdot\ppk
$$
be the material derivative. Then the incompressible Euler equations \eqref{eqs-incomp-Euler} with initial-boundary conditions \eqref{cond-in-bdy} is converted into
\begin{equation}\label{eqs-incomp-Euler-fix}
  \left\{\begin{aligned}
  &D_t^\vp v + \ppk q =  0,\quad &\text{in}\; \Omega,\\
 & \ppk\cdot v =  0,\quad &\text{in}\;\Omega,\\
  &\p_t \psi =  v\cdot N,\quad&\text{on}\; \Gamma_\t,\\
  &q =  -\sigma\TP \cdot \left(\tfrac{\TP \psi}{\sqrt{1+|\TP \psi|^2}}\right), \quad&\text{on}\;\Gamma_\t,\\
 & v\cdot n =  0,\quad&\text{on}\;\Gamma_\b,\\
 &(v,\psi)\big|_{t=0} =  (v_0,\psi_0).\\
  \end{aligned}\right.
\end{equation}
Also, note that we can express
\begin{align}\label{eq-Dtvp'}
\Dtp = \p_t + \vb\cdot \TP + \frac{1}{\p_3\vp}(v\cdot \N-\p_t\vp)\p_3
\end{align}
after invoking \eqref{op-diff-fix},
where $\vb:= (v_1, v_2)$ and $\N:= (-\p_1\vp, -\p_2\vp, 1)^T$. It can be seen that the kinematic boundary condition $\p_t\psi =v\cdot N$ on $\Gamma_{\t}$ indicates that
$$
\Dtp|_{\Gamma_{\t}} = \p_t + \vb\cdot \TP.
$$
Moreover, since $v\cdot n=0$ on $\Gamma_{\b}$, $\p_3\vp|_{\Gamma_{\b}}=1$, and $\p_t \vp|_{\Gamma_{\b}}=0$, we have
$$
\Dtp|_{\Gamma_{\b}} = \p_t+\vb\cdot \TP.
$$
In other words, $\Dtp|_{\p\Omega} \in \mathcal{T} (\p\Omega)$, where $\mathcal{T}(\p\Omega)$ is the tangential bundle of $\p\Omega$.  Also,  by restricting the momentum equation (the first equation of \eqref{eqs-incomp-Euler-fix}) on $\Gamma_{\b}$ and taking the normal component, one has
\begin{align}
n\cdot \p q\big|_{\Gamma_{\b}}=0. \label{cond-in-fixedbdy}
\end{align}
\begin{nota}[Norms]\label{nota norm} We adopt the following norms in the sequel of this manuscript.
\begin{itemize}
\item[i.]  ($H^s$-Sobolev norms) $\|\cdot\|_{s}:=\|\cdot \|_{H^s(\Omega)}$, $|\cdot|_s := \|\cdot\|_{H^s(\Gamma_{\t})}$.
\item[ii.] ($L^\infty$-based Sobolev norms) $\|\cdot\|_{\infty}:=\|\cdot\|_{L^\infty(\Omega)}$, $\|\cdot\|_{W^{1,\infty}}:=\|\cdot\|_{W^{1,\infty}(\Omega)}$, $|\cdot|_{\infty}:=\|\cdot\|_{L^\infty(\Gamma_{\t})}$, $|\cdot|_{W^{1,\infty}}:=\|\cdot\|_{W^{1,\infty}(\Gamma_{\t})}$.
\item[iii.] (H\"older norms) $|\cdot|_{C^k} :=\|\cdot\|_{C^k(\Gamma_{\t})}$.
\end{itemize}
\end{nota}

\subsection{Main Results}
The local existence theorem for the free-boundary incompressible Euler equations can be stated as follows: Let $(v_0, \psi_0)\in H^s(\Omega)\times H^{s+1}(\Gamma_{\t})$ for some fixed $s\geq 3$. Then there exists a $T_0>0$, depends on $\|v_0\|_{s}$ and $|\psi_0|_{s+1}$, such that the equations \eqref{eqs-incomp-Euler-fix} have a unique solution in
$$
C([0,T_0]; H^s(\Omega)\times H^{s+1} (\Gamma_{\t})).
$$
We refer to \cite{DK, Shatah-Zeng-2008CPAM-1,Shatah-Zeng-2008CPAM-2,Shatah-Zeng-2011ARMA} for the local well-posedness of the system \eqref{eqs-incomp-Euler-fix}. Also, we can retrieve the local existence from \cite[Theorem 1.1]{Luo-Zhang-2022} after taking the incompressible limit (with fixed $\sigma>0$).

\begin{theorem}\label{thm-weak-CN}
  Let $(v(t),\psi(t))\in H^{s}(\Omega)\times H^{s+1}(\Gamma_{\t})$, $s>\frac{9}{2}$, be the solution of \eqref{eqs-incomp-Euler-fix} described above. Let
  \begin{equation}\label{T*}
    T^* = \sup\left\{T>0\,\big|\,(v(t),\psi(t)) \text{ can be continued in the class}\;C([0,T]; H^3(\Omega)\times H^{4} (\Gamma_{\t}))\right\}.
  \end{equation}
  If $T^*<+\infty$, then at least one of the following three statements hold:
  \begin{itemize}
  \item [a.] \begin{equation}\label{cond-K1}
    \lim_{t\nearrow T^*}\K(t) = + \infty,
  \end{equation}
  where
  \blue{\begin{equation*}
  \begin{aligned}
 \K(t) :=& \K_1(t)+\K_2(t),\\
  \K_1(t) := |\psi(t)|_{C^3} + |\psi_t(t)|_{C^3} + |\psi_{tt}(t)|_{1.5}, &\quad \K_2(t):=\vbi+|\vb(t)|_{\infty}, 
  \end{aligned}
  \end{equation*}}
  \item [b.]\begin{equation}\label{cond-v}
    \int_{0}^{T^*}\|v(t)\|_{W^{1,\infty}}\d t = +\infty,
  \end{equation}
  \item [c.] \begin{equation}\label{cond-geo}
  \lim_{t\nearrow T^*} \left (\frac{1}{\p_3 \vp(t)} + \frac{1}{b-|\psi(t)|_{\infty}} \right)=+\infty,
  \end{equation}
  or turning occurs on the moving surface boundary.
  \end{itemize}
  \blue{Moreover, if $\vb(t)$, $\p_1 \vb(t)$, and $\p_2 \vb(t)$ are continuous on $\Omega$, then $\vbi$ in $\K_2(t)$ can be dropped. }
\end{theorem}
\begin{remark}
\blue{The last sentence in Theorem \ref{thm-weak-CN} indicates that, if $v(t)$ is a smooth solution (as opposed to a $H^s(\Omega)$-solution), then $\K_2(t)$ is reduced to $|\vb(t)|_{\infty}$. }
\end{remark}
\begin{remark}
The first term in $\K_1(t)$, i.e.,  $ |\psi(t)|_{C^3}$ controls the second fundamental form $\Theta$ of the moving boundary in $C^1(\Gamma_{\t})$, where $\Theta := \TP (\frac{N(t)}{|N(t)|})$ which contributes to $\TP^2 \psi(t)$ in the leading order. Moreover, the second and third terms in $\K_1(t)$, i.e., $|\psi_t(t)|_{C^3}$ and $|\psi_{tt}(t)|_{1.5}$,  control respectively the velocity and acceleration of the moving boundary.
\end{remark}
\begin{remark}
The quantities on the LHS of \eqref{cond-geo} are required to be finite to continue the solution. Recall that we need $\p_3\vp>0$ to ensure the mapping $\Phi(t,\cdot):\Omega\to\DD_t$ is  invertible. In addition, $b-|\psi(t)|_{\infty}>0$ ensures that the upper moving boundary is strictly above the fixed bottom.
\end{remark}

Next, we show that \eqref{cond-v} can be relaxed to $\int_0^{T^*}\|\omp(t)\|_{\infty}\d t=+\infty$, where $\omp:= \pp\times v$.
\begin{theorem}\label{thm-strong-CN}
  Let $T^*$ and $\K(t)$ be as in Theorem \ref{thm-weak-CN}. If $T^*<+\infty$, then at least one of the following three  statements hold:
  \begin{itemize}
  \item [a.] \begin{equation}\label{cond-K1'}
    \lim_{t\nearrow T^*}\K(t) = + \infty,
  \end{equation}
  \item [b'.]   \begin{equation}\label{cond-omega-1}
    \int_{0}^{T^*}\|\omp(t)\|_{L^\infty}\d t = +\infty,
  \end{equation}
  \item [c.]\begin{equation}\label{cond-geo'}
  \lim_{t\nearrow T^*} \left (\frac{1}{\p_3 \vp(t)} + \frac{1}{b-|\psi(t)|_{\infty}} \right)=+\infty,
  \end{equation}
  or turning occurs on the moving surface boundary.
  \end{itemize}
  \blue{Moreover, if $\vb(t)$, $\p_1 \vb(t)$, and $\p_2 \vb(t)$ are continuous on $\Omega$, then $\vbi$ in $\K_2(t)$ can be dropped. }
\end{theorem}
\begin{remark}\label{kt last}
Theorem \ref{thm-strong-CN} can be regarded as a generalization of the classical results of Beale-Kato-Majda \cite{BKM-1984CMP} to the free-boundary Euler equations. Specifically,  if $\psi_t=v\cdot N=0$ on $\Gamma_{\t}$, then the moving surface boundary becomes fixed; in other words, $\psi=\psi(x')$ becomes time-independent.
As a consequence,
the control norm $|\psi|_{C^3}$ in $\K_1$ reduces to a non-negative constant, whereas $|\psi_t|_{C^3}=|\psi_{tt}|_{1.5}=0$. Moreover, if $\psi_t =0$ on $\Gamma_{\t}$, the control norms  in $\K_2$ would not even appear. 
Lastly, both $\p_3\vp$ and $b-|\psi|_{\infty}$ are automatically bounded from below by a positive constant.
\end{remark}

\begin{remark}
We require $(v(t), \psi(t))\in H^s(\Omega)\times H^{s+1}(\Gamma_{\t})$, $s>\frac{9}{2}$ in Theorems \ref{thm-weak-CN} and \ref{thm-strong-CN},  but the space of continuation is merely $H^3(\Omega)\times H^4(\Gamma_{\t})$. The loss of regularity is owing to $|\psi_t|_{C^3}=|v\cdot N|_{C^3}$ in $\K(t)$, which cannot be controlled by $E(t)$. We can prove an alternative breakdown criterion in which $|\psi_t|_{C^3}$ is replaced by $|\psi_t|_{C^2}+|\psi_t|_3$, and the latter can be bounded by the energy ties to the local existence in $H^3(\Omega)\times H^4(\Gamma_{\t})$. We devote Section \ref{sect. lossless} to discuss the details.
\end{remark}

\subsection{History and Background}
The study of the free-boundary problems in Euler equations has blossomed over the past three decades. In the case without surface tension (i.e., $\sigma=0$), the first breakthrough came in Wu \cite{Wu1997LWP, Wu1999LWP}, where the local well-posedness (LWP) is established assuming the flow is \textit{irrotational}, under the Rayleigh-Taylor sign condition
\begin{equation}\label{RT sign}
- \nab_{N} p \geq c>0,\quad\text{on}\,\,\p\DD_{t,\t}.
\end{equation}
It is known that the Rayleigh-Taylor sign condition serves as an essential stability condition on the moving surface boundary to ensure the LWP when $\sigma=0$. Otherwise, Ebin \cite{Ebin1987ill} showed that \eqref{eqs-incomp-Euler}--\eqref{cond-in-bdy} is \textit{ill-posed} when $\sigma=0$ if \eqref{RT sign} is violated. We further remark that there are numerous results concerning the long-term well-posedness for the free-boundary incompressible Euler equations with small and \textit{irrotational} data, see, e.g., \cite{Alazard2013GWP, CCG,  Deng2017wwSTGWP, GMS2012GWP, HIT2016ww, IT2016ww2, IP, XW,  Wu2009GWP, Wu2011GWP}.
In the rotational case, Christodoulou--Lindblad \cite{CL2000priori} established the \textit{a priori energy estimate} for \eqref{eqs-incomp-Euler}--\eqref{cond-in-bdy} with $\sigma=0$, and the LWP  was proved by Lindblad \cite{Lindblad2003LWP} using the Nash-Moser iteration and by Zhang--Zhang \cite{ZhangZhang} using the classical energy approach. On the other hand, when $\sigma>0$, \blue{the LWP (as well as the a priori estimate that ties to LWP) for this model was proved independently by Coutand--Shkoller \cite{CS, CS2010},  Disconzi--Kukavica \cite{DK}, Disconzi--Kukavica--Tuffaha \cite{DKT}, Kukavica--Tuffaha--Vicol \cite{KTV},  and Shatah--Zeng \cite{Shatah-Zeng-2008CPAM-1, Shatah-Zeng-2008CPAM-2, Shatah-Zeng-2011ARMA}. Also, Kukavica--Ozanski \cite{KO} studies the LWP with localized $H^{2+\delta}$-vorticity near the free boundary. }

Moreover, there are available results (e.g., \cite{Dan, WangZhang, WZZZ}) concerning the breakdown criterion for the free-boundary Euler equations when $\sigma=0$. Particularly, using paradifferential calculus,  the authors of \cite{WangZhang, WZZZ} proved that, for $T<T^*$, 
\begin{align}
\sup_{t\in[0,T]}\left(\|\mathcal{H}(t)\|_{L^p\cap L^2(\p\DD_{t,\t})}+\|u(t)\|_{W^{1,\infty}(\DD_t)}\right) <+\infty,\quad p\geq 6,\label{Zhifei}\\
\inf_{(t,y')\in \p\DD_{t,\t}}-\nab_N p (t,y')\geq c>0, \label{RT sign criter}
\end{align}
together with a condition analogous to \eqref{cond-geo}.
 Note that \eqref{Zhifei} depends on the boundedness of $\sup_{t\in[0,T]}\|u(t)\|_{W^{1,\infty}(\DD_t)}$, which is stronger than $\int_0^T\|\omega(t)\|_{L^\infty(\DD_t)}\d t<\infty$, where $\omega:=\nab\times u$. Apart from this, \eqref{RT sign criter} is imposed to avoid the Rayleigh-Taylor breakdown described in \cite{CCFGL}.
Recently, Ginsberg \cite{Dan} proved an alternative breakdown criterion by adapting the method of \cite{CL2000priori}, which states that
if $T<T^*$, then
\begin{align}\label{Dan criter}
\int_0^T \Big(\|\omega(t)\|_{L^\infty(\DD_t)}^2+\|\nab u(t)\|_{L^\infty(\p\DD_{t,\t})}+\|\mathcal{N}(u|_{\p\DD_{t,\t}})\|_{L^\infty(\p\DD_{t,\t})}+\|\nab_ND_tp(t)\|_{L^\infty(\p\DD_{t,\t})}\Big)\d t<\infty.
\end{align}
 Here, $\mathcal{N}$ denotes the Dirichlet-to-Neumann operator.  On the other hand, Julin--La Manna \cite{JLa} studied the a priori estimates for the motion of a charged liquid droplet in Eulerian coordinates with $\sigma>0$. As a by-product, they show if $T<T^*$, then
\begin{align}\label{Lamanna criter}
\sup_{t\in[0,T]} \left(|\psi(t)|_{C^{1,\gamma}}+|\mathcal{H}(t)|_{L^1(\p\DD_{t,\t})}+\|\nab u(t)\|_{L^\infty(\DD_t)}+|u_N(t)|_{H^2(\p\DD_{t,\t})}\right)<+\infty,
\end{align}
where the $C^{1,\gamma}$-norm of $\psi$ is expected to be sharp. Nevertheless, compared with Remark \ref{kt last}, it appears to be hard to further reduce either of the aforementioned breakdown criteria to $\int_0^T\|\omega(t)\|_{L^\infty(\DD_t)}\d t<\infty$ if $\p\DD_{t,\t}$ becomes fixed.

\subsection{What is New?}
In this manuscript, we demonstrate a new breakdown criterion for the free-boundary Euler equations when $\sigma>0$. Specifically, with the help of some carefully chosen control norms with explicit physical background on the moving boundary, we can reduce our breakdown criterion to
\blue{the classical Beale-Kato-Majda criterion in a bounded, simply-connected domain, which was shown in \cite{Ferrari-1993CMP}}
 if the kinematic boundary condition on the moving surface boundary is reduced to the slip boundary condition (Remark \ref{kt last}). Moreover,
\blue{if $T^*<+\infty$ and conditions (a) and (c) in Theorem \ref{thm-strong-CN} do not occur, then $\int_0^{T^*} \|\omega(t)\|_{L^\infty(\DD_t)}\,dt=+\infty$. }
  On the other hand, this implies that if $\omega(0)=\mathbf{0}$, then the 3D free-boundary Euler equations with surface tension can blow up only on the moving surface boundary caused by either condition (a) or (c) in Theorem \ref{thm-strong-CN}.

\subsection{Organization} This manuscript is organized as follows. In Section \ref{sect. auxiliary}, we introduce some fundamental results that will be frequently used in our analysis. Apart from this, we provide an overview of the proof of the main theorems in Subsection \ref{subsect 2.5}. Sections \ref{sect. thm 1.1} and \ref{sect. proof thm 1.2} are devoted respectively to prove Theorem \ref{thm-weak-CN} and \ref{thm-strong-CN}. Finally, in Section \ref{sect. lossless}, we provide an alternative criterion with modified control norms without regularity loss.

\subsection{A List of Notations} Apart from the derivatives in Notation \ref{nota deriv} and norms in Notation \ref{nota norm}, we itemize below a list of frequently used notations in this manuscript.
\begin{itemize}
\item $v=u\circ\Phi$, $q=p\circ \Phi$, and $\omp=\pp\times v$. Also, $\omp = \omega\circ\Phi$, where $\omega=\nab\times u$.
\item Let $\mathcal{T}$ be a differential operator. Then $[\mathcal{T}, f]g =\mathcal{T}(fg)-f\mathcal{T}g $, and $[\mathcal{T}, f, g]=\mathcal{T}(fg)-g\mathcal{T}f-f\mathcal{T}g$.
\item We denote by $P=P(\cdots)$ a generic non-negative function in its arguments, and by $C=C(\cdots)$ a positive constant.
\end{itemize}

\subsection*{Acknowledgment}
 The authors would like to thank Francisco Gancedo, Yao Yao, and Junyan Zhang for sharing their insights.
Also, the authors thank the anonymous referee for helpful comments that improved the quality of the manuscript. 

%=============================================
\section{Some Auxiliary Results and an Overview of Our Strategy} \label{sect. auxiliary}
\subsection{The Change of Coordinates $\Phi(t,\cdot)$}\label{sect. 2.1}
Since $\pp_a = \p_a -\frac{\p_a\vp}{\p_3\vp}\p_3$, $a=1,2$, and $\pp_3= \frac{1}{\p_3\vp}\p_3$, we have
\begin{equation}\label{eq-ppk-p}
\begin{pmatrix}
\pp_1\\
\pp_2\\
\pp_3
\end{pmatrix}
=
\begin{pmatrix}
1&0&-\frac{\p_1 \varphi}{\p_3\varphi}\\
0&1&-\frac{\p_2\varphi}{\p_3\varphi}\\
0&0&\frac{1}{\p_3\varphi}
\end{pmatrix}
\begin{pmatrix}
\p_1\\
\p_2\\
\p_3
\end{pmatrix}.
  \end{equation}
In other words, let
\begin{equation} \label{A}
\A := \begin{pmatrix}
1&0&-\frac{\p_1 \varphi}{\p_3\varphi}\\
0&1&-\frac{\p_2\varphi}{\p_3\varphi}\\
0&0&\frac{1}{\p_3\varphi}
\end{pmatrix}^T
\end{equation}
be the cofactor matrix associated with $\Phi$.
Then for each $i=1,2,3$,
\begin{equation}
\pp_i = \A_i^j\p_j.
\end{equation}
The Einstein summation convention is used here and in the sequel on repeated upper and lower indices.  Also, $\A$ is invertible as long as $\p_3 \varphi>0$, where
\begin{equation} \label{A inverse}
\A^{-1} =
\begin{pmatrix}
1&0&\p_1\vp\\
0&1&\p_2 \vp\\
0&0&\p_3\vp
\end{pmatrix}^T,
\end{equation}
  and
  \begin{equation}
  \p_i = (\A^{-1})_i^j \pp_j.
  \end{equation}
 \subsection{The Sobolev and H\"older Norms of $\vp$}
 In light of \eqref{eq-phi}, we can reduce both the interior Sobolev and H\"older norms of $\vp$ to the associated boundary norms of $\psi$. Particularly, we have
 \begin{align*}
\p_t \vp = \chi\p_t\psi, \quad \TP \vp = \chi \TP\psi,\quad \p_3\vp =1+\chi' \psi.
 \end{align*}
 Invoking \eqref{ineq-chi},
 this implies:
  \begin{align}\label{vp to psi}
  \begin{aligned}
 & \|\vp\|_{C^k(\Omega)} \leq  C(|\psi|_{C^k}+1), \quad \|\p_t\vp\|_{C^k(\Omega)} \leq C|\p_t\psi|_{C^k},  \quad k=0,1,2,3,\\
  &\|\vp\|_s \leq C(|\psi|_s+1), \quad \|\p_t \vp\|_{s} \leq C|\p_t\psi|_s, \quad 0\leq s\leq 3.
  \end{aligned}
  \end{align}
  These estimates will be adapted frequently and silently in the rest of this manuscript.

  \subsection{The Hodge-type Div-Curl Estimate} The following Hodge-type elliptic estimates play a crucial role while bounding $\|v\|_3$ and $\|\p q\|_2$ in the upcoming sections. Here, we denote by $\p q$ the vector $(\p_1 q, \p_2 q, \p_3 q)^T$.
  \begin{lemma}\label{Lem-Hodge}
  For any sufficiently smooth vector field $X$ and integer $s\geq 1$, there exist $C_0:= C_0(|\psi|_{C^s})>0$ such that
  \begin{equation}\label{es-Hodge-1}
    \|X\|_s^2 \leq C_0(|\psi|_{C^s})\left(\|\ppk\cdot X\|_{s-1}^2 + \|\ppk\times X\|_{s-1}^2 + \|\overline{\partial}^s X\|_0^2+\|X\|_0^2\right),
  \end{equation}
  where $\displaystyle \overline{\partial}^s X = \sum_{|\alpha|=s}\overline{\partial}^\alpha X$. Also, for $s>1.5$, there exists $C_1:= C_1(|\psi|_{C^{s+1}})>0$, so that
  \begin{equation}\label{es-Hodge-2}
  \|X\|_s^2 \leq C_1(|\psi|_{C^{s+1}})\left(\|\ppk\cdot X\|_{s-1}^2 + \|\ppk\times X\|_{s-1}^2 + |X\cdot N|_{s-0.5}^2+\|X\|_0^2\right),
  \end{equation}
  provided that $X\cdot n=0$ on $\Gamma_{\b}$.
\end{lemma}
\begin{proof}
The estimate \eqref{es-Hodge-2} is proved in \cite{cheng2017solvability}, while we refer to Appendix \ref{appendix C} for the proof of \eqref{es-Hodge-1}.
\end{proof}

\subsection{An Overview of Our Strategy} \label{subsect 2.5}
A crucial step to prove Theorem \ref{thm-strong-CN} via Theorem \ref{thm-weak-CN} is to establish an energy estimate for
$$E(t):=\|v(t)\|_3^2+\sigma |\psi|_4^2$$
that take the following form:
\begin{align}\label{strategy energy}
 E(t)+\sqrt{E(t)}\leq P(c_0^{-1}, \K_1(t))E(0) &+ \blue{\int_0^t P(c_0^{-1}, \K_1(\tau))\left([\,1+ \|v(\tau)\|_{W^{1,\infty}}\,][\,1+ |\vb(\tau)|_{\infty}\,](\,E(\tau) + \sqrt{E(\tau)}\,)\right)\d \tau}\nonumber\\
 &\blue{+ \int_0^t P(c_0^{-1}, \K_1(\tau))\left([\,1+ |\vb(\tau)|_{W^{1,\infty}}\,](\,E(\tau) + \sqrt{E(\tau)}\,)\right)\d \tau,}
\end{align}
where $P(\cdots)$ denotes a non-negative continuous function in its arguments. \blue{Here, the second line in \eqref{strategy energy} drops if $\vb(t)$, $\p_1 \vb(t)$, and $\p_2 \vb(t)$ are continuous on $\Omega$. }

\blue{It is important to notice that the first line in the energy estimate \eqref{strategy energy} must be linear in both $E(\tau)+\sqrt{E(\tau)}$ and $\|v(\tau)\|_{W^{1,\infty}}$ under the time integral. Once \eqref{strategy energy} is done, we can prove Theorem \ref{thm-strong-CN} by adapting the $L^{\infty}$-Calderon-Zygmund-type estimate in a bounded, simply connected $C^3$-domain (i.e., Lemma \ref{thm Ferrari}) to $\DD_t$. Here, we apply the $L^\infty$-Calderon-Zygmund estimate to the modified velocity field $V$ that verifies the slip boundary condition on the moving surface boundary $\p\DD_{t, \t}$. This can be done by considering $V=u-\tilde{u}$ with $\tilde{u}=\nab\xi$, where $\xi$ is harmonic in $\DD_t$, and satisfying the Neumann boundary condition $\nab_{N} \xi = u\cdot N$ on $\p\DD_{t,\t}$. }
\subsubsection{Proof of \eqref{strategy energy}: }
\blue{The rest of this section is devoted to discussing the proof of \eqref{strategy energy} in succinct steps.\\
\\
\textbf{Step 1: The div-curl analysis}} \\
We adapt \eqref{es-Hodge-1} in Lemma \ref{Lem-Hodge} to decompose $\|v\|_3^2$ into $\|\omp\|_2^2$ and $\|\TP^3 v\|_0^2$ at the leading order. The curl part $\|\omp\|_2^2$ can be controlled straightforwardly by invoking the evolution equation of $\omp$. Moreover, a large portion of Section \ref{sect. thm 1.1} is devoted to control $\|\TP^3 v\|_0^2$ by considering $\TP^3$-differentiated \eqref{eqs-incomp-Euler-fix}. Note that the commutator $[\TP^3, \pp]$ yields a top order term consisting of $4$ spatial derivative on $\vp$. However, we can avoid this by considering the so-called Alinhac's good unknowns of $v$ and $q$, i.e.,
$$\V = \TP^3 v -\p_3^\vp v \TP^3\vp, \quad \Q = \TP^3 q -\p_3^\vp q \TP^3\vp,
$$
 and then obtain an estimate for $\V$ in $L^2(\Omega)$ instead. We need to employ the structure of the equations verified by $\V$ and carefully designed control norms in $\K_1(t)$ to obtain the required linear structure in \eqref{strategy energy}. Specifically, it is helpful to control $\TP^3\vp$ in $L^\infty(\Omega)$ to ensure the linear structure required by \eqref{strategy energy}, where $\|\TP^3\vp\|_{\infty}$ can then be reduced to $|\psi|_{C^3}$ by  \eqref{vp to psi}.\\
\\
\blue{\textbf{Step 2: The tangential energy estimate via good unknowns}\\
This is the most important intermediate step that leads to \eqref{strategy energy}. In particular, we prove that
\begin{align}\label{strategy tangential}
\frac{d}{dt}\left(\|\V(t)\|_0^2+\sigma |\psi|_4^2\right)&\leq P(c_0^{-1}, \K_1(t))\left([\,1+ |\vb(t)|_{\infty}\,][\,1+ \|v(t)\|_{W^{1,\infty}}\,](\,E(t) + \sqrt{E(t)}\,)\right)\nonumber\\
&+P(c_0^{-1}, \K_1(t))\left([\,1+ |\vb(t)|_{W^{1,\infty}}\,](\,E(t) + \sqrt{E(t)}\,)\right).
\end{align}
We establish \eqref{strategy tangential} by testing the higher-order Euler equations (i.e., \eqref{eqs-AGU}) with $\V$ and then integrating in $\Omega$ with respect to $\p_3\vp \,dx$.  The most difficult term generated in this process is the boundary integral
$$
-\int_{\Gamma_{\t}}\Q (\V\cdot N)\,dx'.
$$
We have
\begin{equation}\label{strategy kinematic BC}
\V\cdot N= \p_t\TP^3\psi  + \overline{v}\cdot\TP\left(\TP^3\psi\right)+\cdots, \quad \text{on}\,\,\Gamma_{\t},
\end{equation}
which is obtained by taking $\TP^3$ to $\p_t \psi = v\cdot N$. Here and in the sequel, we employ $\cdots$ to denote easy-to-control error terms.  Also,
\begin{equation}\label{strategy Q}
\Q=-\sigma\TP^3\left( \TP\cdot \tfrac{\TP\psi}{|N|}\right) -\p_3 q \TP^3\psi, \quad \text{on}\,\,\Gamma_{\t}.
\end{equation}}
In light of \eqref{strategy kinematic BC} and \eqref{strategy Q}, we decompose $-\int_{\Gamma_{\t}}\Q (\V\cdot N)\,dx'$ into
\begin{equation}\label{strategy I1}
\begin{aligned}
-\int_{\Gamma_{\t}}\Q (\V\cdot N)\d x'=\sigma\int_{\Gamma_{\t}}\TP^3\left(\TP\cdot \frac{\TP \psi}{|N|}\right)\p_t\TP^3\psi\d x' +\int_{\Gamma_\t} (\p_3 q)(\TP^3\psi)\overline{v}\cdot \TP\left(\TP^3\psi\right) \d x' \\
+\sigma\int_{\Gamma_{\t}}\TP^3\left(\TP\cdot \frac{\TP \psi}{|N|}\right)\vb\cdot \TP \left(\TP^3\psi\right)\d x'+\cdots.
\end{aligned}
\end{equation}
Integrating $\TP\cdot$ by parts, the first term in \eqref{strategy I1} yields the energy term $-\frac{d}{dt}|\TP^4 \psi|_0^2$, together with an error
\begin{align}\label{strategy error}
-\sigma\int_{\Gamma_\t}\left[\TP^2,\,\frac{1}{|N|}\right]\TP^2\psi\cdot \p_t\TP^3\TP\psi\d x'
\end{align}
 at the leading order. Integrating $\TP$ on $\p_t\TP^3\TP\psi$ by parts, this term can be controlled by $P(|\psi|_{C^3})|\psi_t|_{C^3} E$. Note that we need to assign $\TP^3\p_t\psi$ in $L^\infty$ since it is not part of the energy $E$. Moreover, the second term in \eqref{strategy I1} can be controlled by $P(|\psi|_{C^3})|\vb|_{\infty}\|\p q\|_2\sqrt{E}$. Here, we cannot simply bound $|\vb|_{\infty}$ by $\|v\|_{W^{1,\infty}}$ because there is an extra $\|v\|_{W^{1,\infty}}$ generated by the control of $\|\p q\|_2$. \blue{Also, to control the third term in \eqref{strategy I1},  the quantity $|\TP \vb|_{\infty}$ needs to remain bounded in time, where $\TP \vb$ consists of $\p_1\vb$ and $\p_2 \vb$. We point out here that the control of the third term does not involve $\|\p q\|_2$, and so no additional $\|v\|_{W^{1,\infty}}$ would appear.  
 Thus, we have $|\TP \vb|_{\infty} \leq \|\TP \vb\|_{\infty} \leq \|v\|_{W^{1,\infty}}$ by invoking Lemma \ref{thm-Linfty} provided that $\TP\vb$ is continuous on ${\Omega}$. Thanks to this,  the second line in \eqref{strategy tangential} (and hence the second line in \eqref{strategy energy}) can be dropped. As a consequence, the control norm $\vbi$ in $\K_2(t)$ no longer appears when $v(t)$ is a smooth solution.}
\begin{remark} \label{rmk E ext}
 In fact, $|\TP^3\p_t\psi|_0$ is part of the energy involving time derivatives that ties to the local existence of \eqref{eqs-incomp-Euler-fix}, where
\begin{equation*}
E_{\text{exist}}(t)= \sum_{k=0}^{3}\left(\|\p_t^k v(t)\|_{3-k}^2+\sigma |\p_t^k\psi|_{4-k}^2\right).
\end{equation*}
The local existence theory implies that, if $(v(t), \psi(t))\in C([0, T_0]; H^3(\Omega)\times H^4(\Gamma_{\t})$, then
$$
E_{\text{exist}}(t) \leq C(E_{\text{exist}}(0))
$$
holds $\forall t\in [0,T_0]$.
 However, it is difficult to study the breakdown criterion by employing the energy estimate for $E_{\text{exist}}$ as additional interior control norms involving time derivatives of $v$ must be introduced accordingly. As a consequence, we find that it is extremely difficult to prove Theorem \ref{thm-strong-CN} using $E_{\text{exist}}$.
 \end{remark}

\blue{\textbf{Step 3: Estimation of $\|\p q\|_2$}\\
We need to control $\|\p q\|_2$ while studying the tangential estimate \eqref{strategy tangential}.
  Particularly, we study $\|\p q\|_2$ by employing the elliptic equation verified by $q$ equipped with Neumann boundary conditions:
 \begin{equation}
  \begin{aligned}
  -\triangle^\vp q = & (\ppk v)^{T}: (\ppk v),\quad& \i \;\Omega,\\
   N\cdot \ppk q = & -(\overline{v}\cdot\TP v)\cdot N - \p_t^2 \psi - (\overline{v}\cdot\TP)(v\cdot N),  \quad&\text{on}\;\Gamma_\t,\\
  n\cdot \p q=& 0,\quad&\text{on}\;\Gamma_\b.
  \end{aligned}
\end{equation}
We infer from the boundary condition on $\Gamma_{\t}$ that $|\psi_{tt}|_{1.5}$ is needed while controlling $\|\p q\|_2$. We cannot use the Dirichlet boundary condition
$$
q=\sigma\mathcal{H},\quad \text{on}\,\,\Gamma_{\t}
$$
 here as the norms in $\K(t)$ fail to control $|\mathcal{H}|_{2.5}$. Furthermore, since the source term of the elliptic equation of $q$ is quadratic in $\pp v$, it is natural to expect that we require an additional $\|v\|_{W^{1,\infty}}$ when estimating $\|\p q\|_2$. This, together with the estimate of the second term in \eqref{strategy I1} discussed above, implies that we have to put $|\vb(t)|_{\infty}$ to be part of $\K_2(t)$ to ensure that \eqref{strategy energy} is linear in $\|v\|_{W^{1,\infty}}$. }

 \subsubsection{Comparison with the $\sigma=0$ case:} In addition to the leading order error \eqref{strategy error}, the first term on the RHS of \eqref{strategy I1} also generates:
 \begin{equation}\label{strategy RST}
  \int_{\Gamma_\t} \p_3 q\TP^3\psi \p_t\TP^3\psi \d x'.
 \end{equation}
 Note that if the Rayleigh-Taylor sign condition $-\p_3 q(t)\geq c>0$ holds when $t\in [0,T]$, then \eqref{strategy RST} contributes to
 \begin{align}
 -\frac{d}{dt}  \int_{\Gamma_\t} (-\p_3 q)|\TP^3\psi|^2\d x'-\int_{\Gamma_\t} \p_3 \p_t q|\TP^3\psi|^2\d x'.
 \end{align}
 The first term is the boundary energy under the Rayleigh-Taylor sign condition, and the control of the second term requires $|\p_3 \p_t q|_{L^\infty(\Gamma_{\t})}$ to be included in the control norm. This is related to the last quantity on the LHS of Ginsberg's criterion \eqref{Dan criter}. On the other hand, when $\sigma>0$,  \eqref{strategy RST} can be controlled directly by $|\psi_t|_{C^2}\|\p q\|_2 \sqrt{E}$ after integrating $\TP$ in $\p_t\TP^3 \psi$ by parts. This indicates that the surface tension yields a stronger control on the moving boundary and so $|\p_3 \p_t q|_{L^\infty(\Gamma_{\t})}$ is no longer required as part of the control norms.

 \begin{remark}
 It is well-known that one can study the free-boundary problem \eqref{eqs-incomp-Euler}--\eqref{cond-in-bdy} under the Lagrangian coordinates, which is characterized by the flow map $\eta(t,x)$ satisfying $\p_t \eta(t, x) = u(t, \eta(t,x))$. Nevertheless, obtaining a breakdown criterion parallel to Theorem \ref{thm-strong-CN} is difficult under Lagrangian coordinates. The reason is twofold. First,  there are no estimates analogous to \eqref{vp to psi} available in Lagrangian coordinates. Thus one has to introduce new interior control norms, which are not physical compared with the boundary control norms in $\K(t)$. Second, the surface tension takes a different formulation in Lagrangian coordinates, which is more difficult to study than $\mathcal{H}=-\TP\cdot \left(\frac{\TP\psi}{|N|}\right)$. It is still unclear how to obtain an energy estimate analogous to \eqref{strategy energy} under the Lagrangian setting.
 \end{remark}
%===============================================================================================================
\section{Proof of Theorem \ref{thm-weak-CN}}\label{sect. thm 1.1}
We proceed with the proof by contradiction. Assuming $T^*<+\infty$ and none of the conditions (a), (b), and (c) hold in Theorem \ref{thm-weak-CN}, then we show that the solution $(v(t),\psi(t))$ can be continued beyond $T^*$. The key to establishing this is to prove:
\begin{theorem}\label{thm-weak-CN-middle}
Let $T^*$ and $\K(t)$ be as in Theorem \ref{thm-weak-CN}.
Suppose $T^*<+\infty$, and
 there exist constants $M, c_0>0$, such that
\begin{align}
\sup_{t\in [0, T^*)} \K(t) \leq M, \label{assump Kt}\\
\inf_{t\in [0,T^*)}\p_3 \vp(t) \geq c_0, \label{assump c_0}\\
\inf_{t\in [0, T^*)} (b-|\psi(t)|_{\infty}) \geq c_0.\label{assump c_02}
\end{align}
Let
\begin{equation}\label{eq-energy}
  E(t) = \|v(t)\|_3^2+|\sqrt{\sigma}\psi(t)|_4^2.
\end{equation}
Then
\begin{equation} \label{energy est after Gronwall}
E(t) \leq C(c_0^{-1}, M, E(0))\exp\left( \int_0^{t} C(c_0^{-1}, M)(1+\|v(s)\|_{W^{1,\infty}})\d s\right), \quad \forall t\in [0, T^*).
\end{equation}
\end{theorem}
\blue{Theorem \ref{thm-weak-CN} is an immediate consequence of Theorem \ref{thm-weak-CN-middle}: Suppose that neither condition (a) nor condition (c) hold in Theorem \ref{thm-weak-CN}, then \eqref{assump Kt}--\eqref{assump c_02} must be true. Apart from this, the violation of condition (b) in Theorem \ref{thm-weak-CN} indicates that $\int_0^{T^*} \|v(t)\|_{W^{1,\infty}}\,dt <+\infty$. Now, we infer from \eqref{energy est after Gronwall} that $E(T^*)<+\infty$, and so $(v(t), \psi(t))$ can be continued beyond $T^*$. This contradicts the definition of $T^*$ as in \eqref{T*}.}

%An immediate consequence of \eqref{energy est after Gronwall} is that when condition (b) in Theorem \ref{thm-weak-CN} violates, i.e., $\int_0^{T^*} \|v(t)\|_{W^{1,\infty}}\,dt <+\infty$, then the solution $(v(t), \psi(t))$ can be continued beyond $T^*$, which contradicts the the definition of $T^*$ as in \eqref{T*}.

% All the analyses are based on the Proposition \ref{prop-LWP}, namely, we assume we already have a unique smooth solution of the system \eqref{eqs-incomp-Euler-fix} over $[0, T]$. Then we aim to derive an energy estimate for the system \eqref{eqs-incomp-Euler-fix}.
\subsection{$L^2$-Estimates}
The first step to prove Theorem \ref{thm-weak-CN-middle} is to establish the $L^2$-energy estimate for \eqref{eqs-incomp-Euler-fix}. Taking weighted $L^2$ inner products over $\Omega$ by the first equation in \eqref{eqs-incomp-Euler-fix} with $v$ and using the useful identities \eqref{eq-A2} and \eqref{eq-A3}, we obtain
\begin{equation*}
  \begin{aligned}
  &\int_{\Omega} \Dtp v\cdot v\p_3\vp \d x = \frac{1}{2} \ddt \int_{\Omega} |v|^2 \partial_3 \vp \d x,\\
   &\int_{\Omega} \pp q\cdot v\p_3\vp\d x = - \int_{\Omega}q\left(\pp \cdot v \right) \partial_3 \vp \d x + \int_{\Gamma_\t} q v\cdot N \d x'.
  \end{aligned}
\end{equation*}
Thus we get
\begin{equation}\label{L2-es-1}
  \frac{1}{2} \ddt \int_{\Omega} |v|^2 \partial_3 \vp \d x + \int_{\Gamma_\t} q (v\cdot N) \d x' = 0.
\end{equation}
Noting the boundary conditions in \eqref{eqs-incomp-Euler-fix}, we further use integrating $\TP $ by parts to expand the second term above as
\begin{equation*}
  \int_{\Gamma_\t} q (v\cdot N) \d x' = \frac{1}{2}\ddt\int_{\Gamma_\t}\frac{|\sqrt{\sigma}\TP \psi|^2}{|N|} \d x' - \frac{1}{2}\int_{\Gamma_\t}\partial_t\left(|N|^{-1}\right)|\sqrt{\sigma}\TP \psi|^2\d x'.
\end{equation*}
Plugging it into \eqref{L2-es-1}, we obtain
\begin{equation}\label{L2-es-2}
  \frac{1}{2} \ddt \left\{\int_{\Omega} |v|^2 \partial_3 \vp \d x + \int_{\Gamma_\t}\frac{|\sqrt{\sigma}\TP \psi|^2}{|N|} \d x'\right\} = \frac{1}{2}\int_{\Gamma_\t}\partial_t\left(|N|^{-1}\right)|\sqrt{\sigma}\TP \psi|^2\d x'.
\end{equation}
In light of \eqref{assump Kt}-\eqref{assump c_0}, we infer from \eqref{L2-es-2} that, if $t\in [0, T^*)$, then
\begin{align}
\|v(t)\|_0^2+ |\sqrt{\sigma} \TP \psi(t)|_0^2\leq P(c_0^{-1}, |\psi_0|_{C^3})E(0)+\int_0^t P(c_0^{-1}, |\psi(\tau)|_{C^3}, |\psi_t(\tau)|_{C^3})|\sqrt{\sigma} \TP \psi(\tau)|_0^2\d t. \label{v in L2}
\end{align}
%-----------------------------------------------------------------------------------------------------------------
\subsection{Div-curl Analysis}
We use Lemma \ref{Lem-Hodge} to treat the full $H^s$-norm of $v$.
Applying \eqref{es-Hodge-1} to $v$ with $s=3$, since $\ppk\cdot v =0$, and denoting $\omega^\vp := \ppk\times v$, we obtain
\begin{equation}\label{es-div-curl-ctrl-v}
  \|v\|_3^2 \leq C(|\psi|_{C^3})\left(\|v\|_0^2+\|\omega^\vp\|_2^2 + \|\overline{\partial}^3 v\|_0^2\right).
\end{equation}
This indicates that we need to control $\|\omega^\vp\|_2$ and $\|\TP^3 v\|_0$.

\subsection{Control of the Vorticity $\|\omp\|_2$} We devote this subsection to bound $\|\omp\|_2$.
\begin{lemma}\label{lem-omega-p-v}
 Let $t\in [0, T^*)$.  Then
  \begin{equation}\label{es-vor-H2}
    \begin{aligned}
  \|\omega^\vp(t)\|_2^2 \leq & P\left(c_0^{-1},|\psi_0|_{C^3}\right)E(0)
   +\int_{0}^{t}P(c_0^{-1}, |\psi(\tau)|_{C^3}, |\psi_t(\tau)|_{C^3})\|v(\tau)\|_{W^{1,\infty}}E(\tau) \d \tau.
    \end{aligned}
  \end{equation}
\end{lemma}
\begin{proof}
  By taking $\ppk\times$ to the first equation in \eqref{eqs-incomp-Euler-fix}, we obtain:
\begin{equation}\label{eq-vor}
  \Dtp\omega^\vp  = \omega^\vp\cdot\ppk v.
\end{equation}
Then, we apply $\p^\gamma = \p_1^{\gamma_1}\p_2^{\gamma_2}\p_3^{\gamma_3}$ with $|\gamma| \leq 2$ to \eqref{eq-vor} to acquire:
\begin{equation}\label{eq-vor1}
  D_t^\vp (\p^\gamma\omega^\vp) = -\left[\p^\gamma,\,D_t^\vp\right]\omega^\vp + \p^\gamma\left(\omega^\vp\cdot\ppk v\right).
\end{equation}
By the virtue of \eqref{eq-Dtvp'},
we have
\begin{equation}
  \begin{aligned}
  [\p^\gamma, \Dtp] \omega^\vp =
 -\p_3^\vp \omega^\vp D_t^\vp\p^\gamma\vp + \mathcal{R}(\omega^\vp),
  \end{aligned}
\end{equation}
where for $|\alpha'|=1$ with $\alpha'_j\leq \gamma_j$, $j=1,2,3$,
\begin{equation}
  \begin{aligned}
  \mathcal{R}(\omega^\vp) = & \left[\p^\gamma,\,\overline{v}\right]\cdot \TP \omega^\vp + \p_3^\vp \omega^\vp\left[\p^\gamma,v\right]\cdot \N + \left[\p^\gamma,\frac{1}{\p_3 \vp}(v\cdot \N-\p_t\vp),\p_3 \omega^\vp\right]\\
  &\quad + \left[\p^\gamma,v\cdot \N-\p_t\vp, \frac{1}{\p_3 \vp}\right]\p_3 \omega^\vp - (v\cdot \N-\p_t\vp)\p_3 \omega^\vp \left[\p^{\gamma-\alpha'},\frac{1}{(\p_3\vp)^2}\right]\TP^{\alpha'}\p_3\vp,
  \end{aligned}
\end{equation}
By the virtue of standard Sobolev inequalities in Lemma \ref{lem-B1}, we conclude that
\begin{equation}
\|\text{RHS of \eqref{eq-vor1}}\|_{0}\leq
   P\left(c_0^{-1},|\psi|_{C^3}, |\psi_t|_{C^3}\right)\|v\|_3 \|v\|_{W^{1,\infty}}.
\end{equation}

Invoking \eqref{eq-A3}, and  testing \eqref{eq-vor1} with $\p^\gamma\omega^\vp$, we get
\begin{equation}\label{es-vor-H2-1}
  \begin{aligned}
  \frac{1}{2}\ddt \mathcal{E}(t)^2 \leq P\left(c_0^{-1},|\psi|_{C^3},|\psi_t|_{C^3}\right)\|v\|_3 \|v\|_{W^{1,\infty}}\mathcal{E}(t).
  \end{aligned}
\end{equation}
where
\begin{align*}
\mathcal{E}(t)^2 = \int_{\Omega}|\p^\gamma\omega^\vp|^2\p_3\vp\d x.
\end{align*}
Since
$$
\mathcal{E}(t) = \left (\int_{\Omega} |\p^\gamma \omp|^2 \p_3\vp \d x \right)^{\frac{1}{2}}\leq P(|\psi|_{C^3}) \|v\|_3,
$$
and thus
\begin{align*}
\|v\|_3 \mathcal{E}(t) \leq P(|\psi|_{C^3})E(t).
\end{align*}
Then
\eqref{es-vor-H2-1} yields
\begin{align}
\mathcal{E}(t)^2  \leq \mathcal{E}(0)^2 +\int_0^t P\left(c_0^{-1},|\psi(\tau)|_{C^3},|\psi_t(\tau)|_{C^3}\right) \|v(\tau)\|_{W^{1,\infty}} E(\tau)\d \tau,
\end{align}
which leads to \eqref{es-vor-H2} as $\p_3\vp\geq c_0>0$.
\end{proof}
%-------------------------------------------------------
\subsection{Control of the $L^2$-norm of $\p v$} We bound $\|\p v\|_0$ first before treating $\|\TP^3 v\|_0$ appeared on the RHS of \eqref{es-div-curl-ctrl-v}. This quantity plays an important role while studying the bound for $\|\TP^3 v\|_0$ through Alinhac good unknowns.
\begin{lemma}\label{lem-p-v} Let $t\in [0, T^*)$. Then
\begin{equation}\label{es-p-v-L2}
    \begin{aligned}
    \|\p v(t)\|_0^2 \leq &P\left(c_0^{-1},|\psi_0|_{C^3}\right)E(0)+
     \int_{0}^{t}P(c_0^{-1}, |\psi(\tau)|_{C^3}, |\psi_t(\tau)|_{C^3})\left((1 + \|v\|_{W^{1,\infty}})E(\tau) + \|\p q\|_1\sqrt{E(\tau)}\right)\d \tau.
    \end{aligned}
  \end{equation}
\end{lemma}
\begin{proof}
Differentiating the first equation in \eqref{eqs-incomp-Euler-fix} in space and then testing with $\p v$, we then infer from \eqref{eq-A3} that
\begin{equation*}
  \begin{aligned}
  \frac{1}{2}\ddt\int_{\Omega}|\p v|^2\p_3\vp \d x = -\int_{\Omega}\p\left((\overline{v}\cdot\TP v^i) + (v\cdot \N-\p_t\vp)\p_3^\vp v^i\right)(\p v_i)\p_3\vp\d x - \int_{\Omega}(\p \pp_i q) (\p v^i)\p_3\vp\d x.
  \end{aligned}
\end{equation*}
The terms on the RHS can be controlled straightforwardly, which leads to, after integrating in time, that
\begin{equation*}
  \begin{aligned}
  \|\p v(t)\|_0^2 &\leq P\left(c_0^{-1},|\psi_0|_{C^3}\right)\|\p v(0)\|_0^2
  \\&+\int_{0}^{t}P(c_0^{-1}, |\psi(\tau)|_{C^3}, |\psi_t(\tau)|_{C^3})\Big((1 + \|v\|_{W^{1,\infty}})\|v(\tau)\|_3^2 + \|\p q(\tau)\|_1\|v(\tau)\|_3\Big)\d\tau.
  \end{aligned}
\end{equation*}
\end{proof}
%------------------------------------------------------------------------------------------------------------
\subsection{Tangential Energy Estimates with Full Spatial Derivatives}
 We commute $\TP^\alpha$ with \eqref{eqs-incomp-Euler-fix}, where $\alpha = (\alpha_1,\alpha_2)$ with $|\alpha| = 3$ and
\begin{equation}\label{eq-p-alpha}
  \TP^\alpha = \TP_1^{\alpha_1}\TP_2^{\alpha_2}.
\end{equation}

Let $f=f(t,x)$ be a generic smooth function.  For $i=1,2,3$, we have
\begin{equation}\label{eq-tan-cmut-ppk}
  \TP^\alpha\ppk_i f = \ppk_i \left(\TP^\alpha f - \ppk_3 f\TP^\alpha\vp\right) + {\underbrace{R_i^1(f) + \ppk_3\ppk_i f\TP^\alpha\vp}_{R_i^2(f)}},
\end{equation}
with $R^1(f) = (R_1^1,R_2^1,R_3^1)(f)$, and
\begin{equation}\label{eq-R1-f}
  \begin{aligned}
  R_j^1(f) = & -[\TP^\alpha,\frac{\p_j\vp}{\p_3\vp},\p_3 f] - \p_3 f[\TP^\alpha,\p_j\vp,\frac{1}{\p_3\vp}] + \p_3 f \p_j\vp [\TP^{\alpha-\alpha'},\frac{1}{(\p_3\vp)^2}]\TP^{\alpha'}\p_3\vp,\;j=1,2,\\
  R_3^1(f) = & [\TP^\alpha,\frac{1}{\p_3\vp},\p_3 f] - \p_3 f [\TP^{\alpha-\alpha'},\frac{1}{(\p_3\vp)^2}]\TP^{\alpha'}\p_3\vp,
  \end{aligned}
\end{equation}
where $|\alpha'|=1$.

We define
\begin{align}\label{def AGU}
\mathbf{F} := \TP^\alpha f - \ppk_3 f\TP^\alpha\vp
\end{align}
 to be the Alinhac's good unknown associated with $f$, which was first introduced by Alinhac \cite{Alinhac-1989CPDE}.
 Furthermore, recall that we can express $\Dtp$ as
\begin{align}\label{eq-Dtvp}
\Dtp = \p_t + \vb\cdot \TP + \frac{1}{\p_3\vp}(v\cdot \N-\p_t\vp)\p_3, \quad \N = (-\p_1\vp, -\p_2\vp, 1)^T.
\end{align}
Note that $v\cdot\TP^\alpha\N = -(\overline{v}\cdot \TP )\TP^\alpha\vp$, we use \eqref{eq-tan-cmut-ppk} and \eqref{eq-Dtvp} to cast $\TP^\alpha D_t^\vp f$ into
\begin{equation}\label{eq-tan-cmut-Dtvp}
  \begin{aligned}
  \TP^\alpha D_t^\vp f =
   D_t^\vp(\TP^\alpha f - \p_3^\vp f\TP^\alpha \vp) + \underbrace{D_t^\vp\p_3^\vp f\TP^\alpha\vp + {\widetilde{R^3}}(f)}_{=:{R^3}(f)},
  \end{aligned}
\end{equation}
where
\begin{equation}\label{eq-R3-f}
  \begin{aligned}
  \widetilde{R^3}(f) = & \left[\TP^\alpha,\,\overline{v}\right]\cdot \TP f + \p_3^\vp f\left[\TP^\alpha,v\right]\cdot \N + \left[\TP^\alpha,\frac{1}{\p_3 \vp}(v\cdot \N-\p_t\vp),\p_3 f\right]\\
  &\quad + \left[\TP^\alpha,v\cdot \N-\p_t\vp, \frac{1}{\p_3 \vp}\right]\p_3 f - (v\cdot \N-\p_t\vp)\p_3 f \left[\TP^{\alpha-\alpha'},\frac{1}{(\p_3\vp)^2}\right]\TP^{\alpha'}\p_3\vp.
  \end{aligned}
\end{equation}

Now, we define Alinhac's good unknowns
\begin{equation}\label{eq-VQ}
  \V = \TP^\alpha v -\p_3^\vp v \TP^\alpha\vp,\qquad \Q = \TP^\alpha q -\p_3^\vp q \TP^\alpha\vp,
\end{equation}
for $v$ and $q$, respectively.  Thanks to the inequality
\begin{equation}\label{ineq-tan-es-f}
    \|\TP^\alpha v\|_0 \leq \|\V\|_0 + \|\ppk_3 v\TP^\alpha\vp\|_0,
  \end{equation}
  it is not hard to see that we can control $\|\TP^\alpha v\|_0$ through $\|\V\|_0$, while the second term on the RHS of \eqref{ineq-tan-es-f} can be treated by employing Lemma \ref{lem-p-v}.
Taking $\TP^\alpha$ to the system \eqref{eqs-incomp-Euler-fix}, and invoking \eqref{eq-tan-cmut-ppk} and \eqref{eq-tan-cmut-Dtvp}, we obtain the follow system of equations governed by $(\psi, \V, \Q)$:
\begin{equation}\label{eqs-AGU}
  \left\{\begin{aligned}
  &D_t^\vp \V + \ppk \Q =  -R^3(v) - R^2(q),\quad &\text{in}\;\Omega,\\
  &\ppk\cdot \V =  -R_i^2(v^i),\quad &\text{in}\; \Omega,\\
  &\Q =  -\sigma\TP^\alpha\TP\cdot \left(\tfrac{\TP\psi}{\sqrt{1+|\TP\psi|^2}}\right) - \p_3 q\TP^\alpha\psi, \quad&\text{on}\; \Gamma_\t,\\
 & \p_t\TP^\alpha\psi  + \overline{v}\cdot\TP\left(\TP^\alpha\psi\right) - \V\cdot N = \mathcal{S},&\text{on}\; \Gamma_\t,\\
 & \V\cdot n =  0,\quad &\text{on}\; \Gamma_\b,
  \end{aligned}\right.
\end{equation}
where
\begin{equation}\label{eq-S}
  \begin{aligned}
  \mathcal{S} := & \underbrace{(\p_3 v \cdot N)\TP^\alpha\psi + \sum_{\substack{\beta'+\beta^{\prime\prime}= \alpha \\ |\beta'|=1,|\beta^{\prime\prime}|=2}} C_{\alpha}^{\beta'}\left(\TP^{\beta'} v\right)\cdot\left(\TP^{\beta^{\prime\prime}}N\right)}_{=:\mathcal{S}_1} + \sum_{\substack{\beta'+\beta^{\prime\prime}= \alpha \\ |\beta'|=2,|\beta^{\prime\prime}|=1}} C_{\alpha}^{\beta'}\left(\TP^{\beta'} v\right)\cdot\left(\TP^{\beta^{\prime\prime}}N\right)
  \end{aligned}
\end{equation}
\begin{remark}
The boundary condition
\begin{equation}\label{high kinemat cond}
\V\cdot N= \p_t\TP^\alpha\psi  + \overline{v}\cdot\TP\left(\TP^\alpha\psi\right) -\mathcal{S}
\end{equation}
is known to be the ``higher-order kinematic boundary condition" verified by $\V\cdot N$ on $\Gamma_{\t}$.
\end{remark}

The remaining of this subsection is devoted to showing:
\begin{theorem} \label{thm-tangential}
  Let $\V$ be defined as \eqref{eq-VQ} with $|\alpha|=3$. Let $t\in [0, T^*)$. Then
  \blue{\begin{equation}\label{es-AGU-V-2}
  \begin{aligned}
  &\|\V(t)\|_0^2 + |\sqrt{\sigma}\TP^\alpha\,\TP\psi(t)|^2 \\
  & \leq P\left(c_0^{-1},|\psi_0|_{C^3}\right)\left(\|\V(0)\|_0^2 + |\sqrt{\sigma}\TP^\alpha\,\TP\psi(0)|^2\right)\\
  &+ \int_{0}^{t}P(c_0^{-1}, |\psi(\tau)|_{C^3}, |\psi_t(\tau)|_{C^3})\left(\left(1 + |\vb|_{W^{1,\infty}}\right)\left(E(\tau) + \sqrt{E(\tau)}\right) +(1+|\vb|_{\infty}) \|\partial q\|_2\sqrt{E(\tau)}\right)\d \tau\\
  & + \int_{0}^{t}P(c_0^{-1}, |\psi(\tau)|_{C^3}, |\psi_t(\tau)|_{C^3})\,\Big(\|\partial q\|_2 + (1 + \|v\|_{W^{1,\infty}})\sqrt{E(\tau)}\Big)\|\V(\tau)\|_0\d \tau.
  \end{aligned}
\end{equation}}
\blue{Moreover, if $\vb(t)$, $\p_1 \vb(t)$, and $\p_2 \vb(t)$ are continuous on $\Omega$, then
the third line in \eqref{es-AGU-V-2} can be replaced by
\begin{equation}
\int_{0}^{t}P(c_0^{-1}, |\psi(\tau)|_{C^3}, |\psi_t(\tau)|_{C^3})\,\left(\left(1 + \|v\|_{W^{1,\infty}}\right)\Big(E(\tau) + \sqrt{E(\tau)}\Big)+ (1+|\vb|_{\infty})\|\partial q\|_2\sqrt{E(\tau)}\right)\d \tau.
\end{equation}}
\end{theorem}

%\begin{proof}
\subsubsection{Proof of Theorem \ref{thm-tangential}}
We first state some preliminary results that are employed in the proof of Theorem \ref{thm-tangential}. Invoking the definition of $\vp$ in \eqref{eq-phi}, we have
\begin{equation}\label{eq-useful-1}
 \p_3\vp\big|_{\Gamma_{\t}} = 1, \quad \p_3^\vp\big|_{\Gamma_{\t}} = \p_3, \quad  \TP^\alpha \vp\big|_{\Gamma_\t} = \TP^\alpha \psi,\quad \TP^\alpha \vp\big|_{\Gamma_\b} = 0.
\end{equation}
  Testing the first equation in \eqref{eqs-AGU} with $\V$, we obtain:
  \begin{equation}\label{eq-AGU-V}
    \frac{1}{2}\ddt \int_{\Omega}|\V|^2\p_3\vp\d x = \underbrace{-\int_{\Gamma_\t} \Q (\V\cdot N) \d x'}_{I_1} +\underbrace{\int_{\Omega}\Q(\ppk \cdot \V)\d x}_{I_2} \underbrace{-\int_{\Omega}R^3(v)\cdot\V\p_3\vp\d x}_{I_3} \underbrace{-\int_{\Omega}R^2(q)\cdot\V \p_3\vp\d x}_{I_4}.
  \end{equation}
  \textbf{Control of $I_1$:} Invoking the higher-order kinematic boundary condition \eqref{high kinemat cond},  we have
\begin{equation}\label{eq-I1}
  \begin{aligned}
  I_1 = & \underbrace{-\int_{\Gamma_\t} \Q \p_t\TP^\alpha\psi \d x'}_{I_{11}} \underbrace{- \int_{\Gamma_\t} \Q (\overline{v}\cdot\TP )\left(\TP^\alpha\psi\right) \d x'}_{I_{12}} + \underbrace{\int_{\Gamma_\t} (\TP^\alpha q) \mathcal{S}_1 -(\p_3 q \TP^\alpha\psi)\mathcal{S} \d x'}_{I_{13}}\\
  & \qquad+\underbrace{\int_{\Gamma_\t} \TP^\alpha q\sum_{\substack{\beta'+\beta^{\prime\prime}= \alpha \\ |\beta'|=2,|\beta^{\prime\prime}|=1}} C_{\alpha}^{\beta'}\left(\TP^{\beta'} v\right)\cdot\left(\TP^{\beta^{\prime\prime}}N\right) \d x'}_{I_{14}}.
  \end{aligned}
\end{equation}

\paragraph*{Estimation on $I_{11}$:} Invoking the boundary condition of $\Q$ on $\Gamma_{\t}$ (i.e., the third equation in \eqref{eqs-AGU}), we have
\begin{equation}\label{}
  \begin{aligned}
  I_{11} = & \sigma\int_{\Gamma_\t} \TP^\alpha\TP \cdot \left(\tfrac{\TP \psi}{\sqrt{1+|\TP \psi|^2}}\right)\p_t\TP^\alpha\psi \d x' + \int_{\Gamma_\t} \p_3 q\TP^\alpha\psi \p_t\TP^\alpha\psi \d x'\\
  := &\, ST + RST.
  \end{aligned}
\end{equation}
To evaluate $ST$, we will frequently use the following identity:
\begin{equation*}
  \TP_i \left(\frac{1}{|N|}\right) = -\frac{\TP\psi\cdot\TP \,\TP_i\psi}{|N|^3},\quad i=1,2,
\end{equation*}
where $1\leq |N| =\sqrt{1+|\TP \psi|^2}$ denotes the length of normal vector $N=(-\p_1\psi,-\p_2\psi,1)$. Since for $|\alpha'| = 1$ with $\alpha_i^\prime \leq\alpha_i$, we obtain

\begin{equation}\label{TP-alpha-TP-psi}
  \begin{aligned}
  \TP^\alpha \left(\tfrac{\TP\psi}{|N|}\right)  = & \underbrace{ \frac{\TP^\alpha\TP\psi}{|N|} - \frac{\left(\TP\psi\cdot \TP^\alpha\,\TP\psi\right)\TP\psi}{|N|^3}}_{\text{top order}}+ \left[\TP^{\alpha - \alpha'},\,\frac{1}{|N|}\right]\TP^{\alpha'}\TP\psi\\
  &\qquad - \left[\TP^{\alpha - \alpha'},\,\frac{1}{|N|^3}\right]\left((\TP\psi\cdot\TP^{\alpha'}\TP\psi)\TP\psi\right) - \frac{1}{|N|^3}\left[\TP^{\alpha - \alpha'},\,\TP\psi\right]\left(\TP\psi\cdot\TP^{\alpha'}\TP\psi\right).
  \end{aligned}
\end{equation}
This implies, after integrating $\TP\cdot $ by parts in $ST$, that
\begin{equation}\label{}
  \begin{aligned}
  &ST = -\sigma\int_{\Gamma_\t}\left\{\frac{\TP^\alpha\TP\psi}{|N|} - \frac{\left(\TP\psi\cdot \TP^\alpha\,\TP\psi\right)\TP\psi}{|N|^3}\right\}\cdot \p_t\TP^\alpha\TP\psi\d x' -\sigma\int_{\Gamma_\t}\left[\TP^{\alpha - \alpha'},\,\frac{1}{|N|}\right]\TP^{\alpha'}\TP\psi\cdot \p_t\TP^\alpha\TP\psi\d x'\\
  &\qquad\qquad+ \sigma\int_{\Gamma_\t}\left[\TP^{\alpha - \alpha'},\,\frac{1}{|N|^3}\right]\left((\TP\psi\cdot\TP^{\alpha'}\TP\psi)\TP\psi\right)\cdot \p_t\TP^\alpha\TP\psi\d x'\\
  &\qquad\qquad\qquad\qquad + \sigma\int_{\Gamma_\t} \frac{1}{|N|^3}\left[\TP^{\alpha - \alpha'},\,\TP\psi\right]\left(\TP\psi\cdot\TP^{\alpha'}\TP\psi\right)\cdot \p_t\TP^\alpha\TP\psi\d x'\\
  &\quad\,:=\,ST_1+ST_2+ST_3+ST_4.
  \end{aligned}
\end{equation}
Here, $ST_1$ produces the positive energy term contributed by the surface tension, i.e.,
\begin{equation}\label{eq-ST1}
  \begin{aligned}
  ST_1 = & -\frac{1}{2}\ddt\int_{\Gamma_\t}\tfrac{\left|\sqrt{\sigma}\TP^\alpha\TP\psi \right|^2}{(1+|\TP\psi|^2)^{1/2}} -\tfrac{\left|\sqrt{\sigma}\TP\psi\cdot \TP^\alpha\,\TP\psi\right|^2}{(1+|\TP\psi|^2)^{3/2}}\d x' + \frac{\sigma}{2}\int_{\Gamma_\t}\p_t (1+|\TP\psi|^2)^{-1/2}\left|\TP^\alpha\TP\psi \right|^2\d x'\\
  &\qquad - \frac{\sigma}{2}\int_{\Gamma_\t}\p_t (1+|\TP\psi|^2)^{-3/2}\left|\TP\psi\cdot \TP^\alpha\,\TP\psi\right|^2\d x'\\
  &\qquad\qquad -\sigma\int_{\Gamma_\t}(1+|\TP\psi|^2)^{-3/2}\left(\TP\psi\cdot \TP^\alpha\TP\psi\right)\left(\p_t\TP\psi\cdot \TP^\alpha\TP\psi\right)\d x'\\
  := &\, ST_{11} + ST_{12} + ST_{13} + ST_{14}.
  \end{aligned}
\end{equation}
It is clear that
\begin{equation*}
  ST_{12} + ST_{13} + ST_{14} \leq P(|\TP\psi|_\infty)|\p_t\TP\psi|_\infty|\sqrt{\sigma}\TP^\alpha\TP\psi|_0^2,
\end{equation*}
and thus we conclude
\begin{equation}\label{es-ST1}
  ST_1 + \frac{1}{2}\ddt\int_{\Gamma_\t}\tfrac{\left|\sqrt{\sigma}\TP^\alpha\TP\psi \right|^2}{(1+|\TP\psi|^2)^{1/2}} -\tfrac{\left|\sqrt{\sigma}\TP\psi\cdot \TP^\alpha\,\TP\psi\right|^2}{(1+|\TP\psi|^2)^{3/2}}\d x' \leq P\left(|\psi|_{C^1},|\psi_t|_{C^1}\right)|\sqrt{\sigma}\TP^\alpha\TP\psi|_0^2.
\end{equation}

To finish the control of $I_{11}$, it remains to control $ST_2 + ST_3 + ST_4$ and $RST$. For $ST_i,\,i=2,3,4$, we integrate $\TP$ in $\p_t\TP^\alpha\TP\psi$ by parts to get
\begin{equation}\label{es-ST234}
  ST_2 + ST_3 + ST_4 \leq P(|\psi|_{C^3})\, \sum_{k=1}^{3}|\sqrt{\sigma}\TP^{k+1}\psi|_0|\sqrt{\sigma}\psi|_3\, |\p_t\TP^\alpha\psi|_\infty.
\end{equation}
The term $RST$ is controlled by the surface tension energy. Taking $\alpha'\leq \alpha$ with $|\alpha'|=1$, we integrate $\TP^{\alpha'}$ in $\p_t\TP^\alpha\TP\psi$ by parts and then use trace theorem to yield
\begin{equation}\label{es-RST}
  RST \leq |\partial_3 q|_1|\TP^\alpha\psi|_1\,|\partial_t\TP^{\alpha-\alpha'}\psi|_\infty \leq |v\cdot N|_{C^2}\|\partial_3 q\|_{1.5}|\TP^\alpha\psi|_1.
\end{equation}
We collect the estimates \eqref{es-ST1}, \eqref{es-ST234} and \eqref{es-RST} to get $\forall\,t\in [0,T]$,
\begin{equation}\label{es-I11}
  \begin{aligned}
  &I_{11} + \frac{1}{2}\ddt\int_{\Gamma_\t}\tfrac{\left|\sqrt{\sigma}\TP^\alpha\TP\psi \right|^2}{(1+|\TP\psi|^2)^{1/2}} -\tfrac{\left|\sqrt{\sigma}\TP\psi\cdot \TP^\alpha\,\TP\psi\right|^2}{(1+|\TP\psi|^2)^{3/2}}\d x'\\
  &\quad\leq P(|\psi|_{C^3},|\psi_t|_{C^3})\Big\{|\sqrt{\sigma}\TP^\alpha\TP\psi|_0^2+ \sum_{k=1}^{3}|\sqrt{\sigma}\TP^{k+1}\psi|_0|\sqrt{\sigma}\psi|_3 +\|\partial_3 q\|_{1.5}|\TP^\alpha\psi|_1\Big\}.
  \end{aligned}
\end{equation}

\paragraph*{Estimation on $I_{12}$:} We express $I_{12}$ as
\begin{equation}\label{eq-I12}
  \begin{aligned}
  I_{12} = & \int_{\Gamma_\t} \sigma\TP^\alpha\TP\cdot \left(\tfrac{\TP\psi}{\sqrt{1+\TP\psi|^2}}\right) \overline{v}\cdot\TP\left(\TP^\alpha\psi\right) \d x' + \int_{\Gamma_\t} \p_3 q\TP^\alpha\psi\overline{v}\cdot \TP\left(\TP^\alpha\psi\right) \d x'\\
  =: &\,I_{121} + I_{122}.
  \end{aligned}
\end{equation}
The second term $I_{122}$ is bounded by
\begin{equation}\label{es-I122}
  I_{122} \leq |\psi|_{C^3}|\vb|_\infty \|\p_3 q\|_{1} |\TP^\alpha \TP \psi|_0.
\end{equation}
For the first term $I_{121}$, we follow the same process as in the estimate of $I_{11}$.  Integrating $\TP\cdot$ by parts in $I_{121}$ and then invoking \eqref{TP-alpha-TP-psi}, we obtain
\begin{equation}\label{es-I121}
  \begin{aligned}
  I_{121} = & -\sigma\int_{\Gamma_\t} \left\{\frac{\TP^\alpha\TP\psi}{|N|} - \frac{\left(\TP\psi\cdot \TP^\alpha\,\TP\psi\right)\TP\psi}{|N|^3}\right\}\cdot \left\{\TP\overline{v}\cdot\left(\TP\,\TP^\alpha\psi\right) + \overline{v}\cdot \TP\left(\TP\,\TP^\alpha\psi\right)\right\} \d x'\\
  & \qquad - \sigma\int_{\Gamma_\t} \left[\TP^{\alpha - \alpha'},\,\frac{1}{|N|}\right]\TP^{\alpha'}\TP\psi\cdot \left\{\TP\overline{v}\cdot\left(\TP^\alpha\TP\psi\right) + \overline{v}\cdot \TP\left(\TP^\alpha\TP\psi\right)\right\} \d x'\\
   & \qquad + \sigma\int_{\Gamma_\t} \left[\TP^{\alpha - \alpha'},\,\frac{1}{|N|^3}\right]\left((\TP\psi\cdot\TP^{\alpha'}\TP\psi)\TP\psi\right) \cdot \left\{\TP\overline{v}\cdot\left(\TP^\alpha\TP\psi\right) + \overline{v}\cdot \TP\left(\TP^\alpha\TP\psi\right)\right\} \d x'\\
   & \qquad + \sigma\int_{\Gamma_\t} \frac{1}{|N|^3}\left[\TP^{\alpha - \alpha'},\,\TP\psi\right]\left(\TP\psi\cdot\TP^{\alpha'}\TP\psi\right)\cdot \left\{\TP\overline{v}\cdot\left(\TP^\alpha\TP\psi\right) + \overline{v}\cdot \TP\left(\TP^\alpha\TP\psi\right)\right\} \d x'\\
  \leq & P(|\psi|_{C^2}) |\overline{v}|_{W^{1,\infty}}\left|\sqrt{\sigma}\TP^\alpha\TP\psi\right|_0^2 + P(|\psi|_{C^3})|\overline{v}|_{W^{1,\infty}} \sum_{k=1}^{3}|\sqrt{\sigma}\TP^{k+1}\psi|_0\left|\sqrt{\sigma}\TP^\alpha\TP\psi\right|_0.
  \end{aligned}
\end{equation}
Plugging \eqref{es-I121} and \eqref{es-I122} into \eqref{eq-I12}, we get
\blue{\begin{equation}\label{es-I12}
  \begin{aligned}
  I_{12}
  &\leq P(|\psi|_{C^3})|\vb|_{\infty}\|\p_3 q\|_1|\TP^\alpha\TP \psi|_0\\
  &+ P(|\psi|_{C^3})|\vb|_{W^{1,\infty}}\left( \left|\sqrt{\sigma}\TP^\alpha\TP\psi\right|_0^2
  + \sum_{k=1}^{3}|\sqrt{\sigma}\TP^{k+1}\psi|_0\left|\sqrt{\sigma}\TP^\alpha\TP\psi\right|_0\right).
  \end{aligned}
\end{equation}}
\begin{remark} \label{rmk I12}
The estimate for $I_{12}$, in particular, the first term in the RHS of \eqref{es-I12}, yields a structure of the following type:
\begin{equation}\label{special structure}
\blue{P(|\psi|_{C^3}) |\vb|_{\infty}\|\p_3 q\|_1\sqrt{E}.}
\end{equation}
Since $\|\p_3 q\|_1\leq P(c_0^{-1}, |\psi|_{C^3})\|\pp q\|_2$, the estimate \eqref{es-p-q-H2} implies that \eqref{special structure} becomes:
\begin{equation}\label{special structure'}
\blue{P(c_0^{-1}, |\psi|_{C^3}) |\vb|_{\infty}\left(\|v\|_{W^{1,\infty}}E+|\psi_{tt}|_{1.5}\sqrt{E}\right).}
\end{equation}
It is important to see that \eqref{special structure'} depends linearly on $\|v\|_{W^{1,\infty}}$, which eventually leads to \eqref{energy est after Gronwall} in Theorem \ref{thm-weak-CN-middle} provided that
$$\blue{|\vb|_{L^\infty([0, T^*);L^\infty)}\leq M.}$$ Note that this linear structure in $\|v\|_{W^{1, \infty}}$ is essential in the proof of Theorem \ref{thm-strong-CN}.
\end{remark}

\blue{Moreover, if $\vb(t)$ and $\TP  \vb(t)=\left(\p_1 \vb(t), \p_2\vb (t)\right)^T$ are continuous on $\Omega$, then we infer from Lemma \ref{thm-Linfty} that $|\vb|_{W^{1,\infty}} \leq \|\vb\|_{W^{1,\infty}}$.  This allows us to bound $|\vb|_{W^{1,\infty}}$ by $\|v\|_{W^{1,\infty}}$ since $\|\vb\|_{W^{1,\infty}}\leq \|v\|_{W^{1,\infty}}$.  In consequence,
\begin{equation}\label{es-I12 smh}
  \begin{aligned}
  I_{12}
  &\leq P(|\psi|_{C^3})|\vb|_{\infty}\|\p_3 q\|_1|\TP^\alpha\TP \psi|_0\\
  &+ P(|\psi|_{C^3})\|v\|_{W^{1,\infty}}\left( \left|\sqrt{\sigma}\TP^\alpha\TP\psi\right|_0^2
  + \sum_{k=1}^{3}|\sqrt{\sigma}\TP^{k+1}\psi|_0\left|\sqrt{\sigma}\TP^\alpha\TP\psi\right|_0\right).
  \end{aligned}
\end{equation}
}

\paragraph*{Estimation on $I_{13}$:} It remains to control the term $I_{13}$ in \eqref{eq-I1}. As before, we integrate $\TP\cdot$ by parts in the mean curvature term to obtain
\begin{equation}\label{eq-I13}
  I_{13} = \sigma\int_{\Gamma_\t} \TP^\alpha \left(\tfrac{\TP\psi}{\sqrt{1+|\TP\psi|^2}}\right)\cdot \TP \mathcal{S}_1 \d x' - \int_{\Gamma_\t} \partial_3 q\TP^\alpha\psi\mathcal{S}\d x':=I_{131}+I_{132}.
\end{equation}
Note that $\mathcal{S}_1$ contributes to $(\p_3v\cdot N)\TP^\alpha \psi$, and we would like to re-express $\p_3v\cdot N$ in a way that only tangential derivatives are involved.
Since $\pp\cdot v=0$, it holds that
\begin{align*}
\p_3 v^3=-\p_3\vp\p_1 v^1+\p_1\vp \p_3 v^1-\p_3\vp\p_2v^2+\p_2\vp\p_3 v^2,\quad\text{in}\,\,\Omega,
\end{align*}
which becomes, after restricting on $\Gamma_{\t}$, that
\begin{align*}
\p_3v^3 = -\p_1v^1-\p_2v^2+\p_1\psi \p_3 v^1+\p_2\psi\p_3 v^2,\quad\text{on}\,\,\Gamma_{\t}.
\end{align*}
Now, because
$$
\p_3 v\cdot N = -\p_3v^1 \p_1\psi-\p_3 v^2\p_2\psi + \p_3 v^3, \quad \text{on}\,\,\Gamma_{\t},
$$
we obtain
\begin{equation}\label{normal component of p3v}
\p_3 v\cdot N = -\TP\cdot \vb, \quad\text{on}\,\,\Gamma_{\t},
\end{equation}
and thus
\begin{align}\label{TP S1}
\TP \mathcal{S}_1 = -\TP  (\TP \cdot \vb)\TP^\alpha\psi-(\TP \cdot \vb)(\TP^\alpha\TP\psi)+ \sum_{\substack{\beta'+\beta^{\prime\prime}= \alpha \\ |\beta'|=1,|\beta^{\prime\prime}|=2}} C_{\alpha}^{\beta'}\TP\left\{\left(\TP^{\beta'} v\right)\cdot\left(\TP^{\beta^{\prime\prime}}N\right)\right\}.
\end{align}
Here, since $N=(-\TP\psi, 1)^T$, we have
\begin{equation*}
\left(\TP^{\beta'} v\right)\cdot\left(\TP^{\beta^{\prime\prime}}N\right)= -\left(\TP^{\beta'} \vb\right)\cdot\left(\TP^{\beta^{\prime\prime}}\TP\psi\right).
\end{equation*}

We plug the identity \eqref{TP S1} to $I_{131}$ and obtain
\begin{align*}
I_{131} =\sigma \int_{\Gamma_{\t}} \TP^{\alpha}\left(\frac{\TP\psi}{|N|}\right)\cdot \left[-\TP(\TP\cdot \vb)\TP^\alpha\psi-(\TP\cdot \vb)(\TP^\alpha\TP \psi)-\sum_{\substack{\beta'+\beta^{\prime\prime}= \alpha \\ |\beta'|=1,|\beta^{\prime\prime}|=2}} C_{\alpha}^{\beta'}\TP\left\{\left(\TP^{\beta'} \vb\right)\cdot\left(\TP^{\beta^{\prime\prime}}\TP\psi\right)\right\}\right]\d x',
\end{align*}
and thus
\begin{equation}
I_{131}\leq P(|\psi|_{C^3})|\vb|_{2}\sum_{k=1}^3|\sqrt{\sigma}\TP^k\TP \psi|_0+P(|\psi|_{C^3})|\vb|_{W^{1,\infty}}|\sqrt{\sigma}\TP^\alpha\TP\psi|_0\sum_{k=1}^3|\sqrt{\sigma}\TP^k\TP \psi|_0.
\end{equation}
Moreover,
\begin{equation}
I_{132} \leq P(|\psi|_{C^3})|\p_3 q|_0|\vb|_2.
\end{equation}
Then, by combining these two estimates, we have
\begin{equation}\label{es-I13 sob}
  \begin{aligned}
  I_{13}
  \leq P\left(|\psi|_{C^3}\right)\Big((1+|\vb|_{W^{1,\infty}})|\sqrt{\sigma}\psi|_4^2+\|v\|_{2.5}(|\sqrt{\sigma}\psi|_4+\|\p_3 q\|_1)\Big).
  \end{aligned}
\end{equation}
\blue{On the other hand, if $\vb(t)$ and $\TP  \vb(t)=\left(\p_1 \vb(t), \p_2\vb (t)\right)^T$ are continuous on $\Omega$, then parallel to the deduction of \eqref{es-I12 smh}, we have $|\vb|_{W^{1,\infty}} \leq \|\vb\|_{W^{1,\infty}}\leq \|v\|_{W^{1,\infty}}$. Therefore, 
\begin{equation}\label{es-I13 smh}
  \begin{aligned}
  I_{13} \leq   P\left(|\psi|_{C^3}\right)\Big((1+\|v\|_{W^{1,\infty}})|\sqrt{\sigma}\psi|_4^2+\|v\|_{2.5}(|\sqrt{\sigma}\psi|_4+\|\p_3 q\|_1)\Big).
  \end{aligned}
\end{equation}}

\paragraph*{Estimation on $I_{14}$:} We recall that
\begin{equation*}
  I_{14}=\int_{\Gamma_\t} \TP^\alpha q\sum_{\substack{\beta'+\beta^{\prime\prime}= \alpha \\ |\beta'|=2,|\beta^{\prime\prime}|=1}} C_{\alpha}^{\beta'}\left(\TP^{\beta'} v\right)\cdot\left(\TP^{\beta^{\prime\prime}}N\right) \d x'.
\end{equation*}
We use the $H^{-\frac{1}{2}}-H^{\frac{1}{2}}$ duality argument and trace theorem to obtain
\begin{equation}\label{es-I14}
  I_{14} \leq P\left(|\psi|_{C^3}\right)\|\TP^\alpha q\|_0\|v\|_3.
\end{equation}
\blue{Finally, we collect the estimates of $I_{11},\,I_{12},\,I_{13},\,I_{14}$ to conclude that
\begin{equation}\label{es-I1}
\begin{aligned}
  I_1& + \frac{1}{2}\ddt\int_{\Gamma_\t}\tfrac{\left|\sqrt{\sigma}\TP^\alpha\TP\psi \right|^2}{(1+|\TP\psi|^2)^{1/2}} -\tfrac{\left|\sqrt{\sigma}\TP\psi\cdot \TP^\alpha\,\TP\psi\right|^2}{(1+|\TP\psi|^2)^{3/2}}\d x'\\
  \leq&  P\left(|\psi|_{C^3},|\psi_t|_{C^3}\right) \Big\{\left(1 + |\vb|_{W^{1,\infty}}\right)|\sqrt{\sigma}\psi|_4^2 + (1+|\vb|_{\infty})\|\partial q\|_2|\sqrt{\sigma}\psi|_4 \\
  &+|\sqrt{\sigma}\psi|_4\|v\|_3+  \|\partial q\|_2\|v\|_3  \Big\}.
\end{aligned}
\end{equation}}
\blue{In addition, if $\vb(t)$ and $\TP  \vb(t)=\left(\p_1 \vb(t), \p_2\vb (t)\right)^T$ are continuous on $\Omega$, the estimate for $I_{12}$ changes from \eqref{es-I12} to \eqref{es-I12 smh}, while the estimate for $I_{13}$ changes from \eqref{es-I13 sob} to \eqref{es-I13 smh}. As a consequence, we have
\begin{equation}\label{es-I1 smh}
\begin{aligned}
  I_1& + \frac{1}{2}\ddt\int_{\Gamma_\t}\tfrac{\left|\sqrt{\sigma}\TP^\alpha\TP\psi \right|^2}{(1+|\TP\psi|^2)^{1/2}} -\tfrac{\left|\sqrt{\sigma}\TP\psi\cdot \TP^\alpha\,\TP\psi\right|^2}{(1+|\TP\psi|^2)^{3/2}}\d x'\\
  \leq&  P\left(|\psi|_{C^3},|\psi_t|_{C^3}\right) \Big\{\left(1 + \|v\|_{W^{1,\infty}}\right) |\sqrt{\sigma}\psi|_4^2 + (1+|\vb|_{\infty})\|\partial q\|_2|\sqrt{\sigma}\psi|_4 \\
  &+|\sqrt{\sigma}\psi|_4\|v\|_3+  \|\partial q\|_2\|v\|_3  \Big\}.
\end{aligned}
\end{equation}
}
\begin{remark}
  In the case when the moving surface boundary $\p\DD_{t,\t}$ is fixed (e.g., \cite{Ferrari-1993CMP}), $I_1$ is controlled differently and, in particular, the control norms in $\K_2(t)$ no longer appears. To elaborate on this, we first note that we no longer need to introduce Alinhac's good unknowns whenever $\psi=\psi(x')$ is smooth and $t$-independent. As a consequence, $N=(-\TP \psi, 1)$ is also $t$-independent, and the term associated with $I_1$ in \eqref{eq-AGU-V} reads
  \begin{align*}
   -\int_{\Gamma_{\t}} (\TP^\alpha q) (\TP^\alpha v\cdot N)\d x' = -\int_{\Gamma_{\t}} (\TP^\alpha q) \TP^\alpha \underbrace{(v\cdot N)}_{=0}\d x'+\int_{\Gamma_{\t}} (\TP^\alpha q)\left([\TP^\alpha, N]\cdot v\right)\d x'.
  \end{align*}
  Since $\psi$ is smooth, the worst contribution of the last integral is $\int_{\Gamma_{\t}}(\TP^\alpha q)(\TP^{\alpha'} N\cdot \TP^{\alpha-\alpha'} v)\d x'$ with $|\alpha'|=1$, which can be controlled straightforwardly by $C\|\p q\|_2\|v\|_3$, after using the $H^{-\frac{1}{2}}-H^{\frac{1}{2}}$ duality argument and the trace theorem.
  \end{remark}

\noindent{\bf Control of $I_2$.} Note that
\begin{equation}\label{eq-I2}
  I_2 = -\int_{\Omega}\left(\TP^\alpha q -\p_3^\vp q \TP^\alpha\vp\right) \left(R_i^1(v^i) + \ppk_3\ppk_iv^i\TP^\alpha\vp\right)\d x,
\end{equation}
which can be controlled directly by Cauchy-Schwarz inequality:
\begin{equation}\label{es-I2}
  I_2 \leq P\left(c_0^{-1},|\psi|_{C^3}\right)\,\|\p q\|_2\|v\|_3.
\end{equation}

\noindent {\bf Control of $I_3$.} Now we turn to control $I_3$. Recall the definition \eqref{eq-tan-cmut-Dtvp} of remainder $R^3(\cdot)$. We use the Cauchy-Schwarz inequality to get
\begin{align*}
  I_3 \leq \left(\int_{\Omega}|\V|^2\p_3\vp\d x\right)^{1/2}\left(\int_{\Omega}\left|D_t^\vp\p_3^\vp v \TP^\alpha\vp + \widetilde{R^3}(v)\right|^2\p_3\vp\d x\right)^{1/2}.
\end{align*}
To evaluate the second factor on the RHS above, we expand $D_t^\vp\p_3^\vp v$ as
\begin{align*}
  D_t^\vp\p_3^\vp v = & \frac{1}{\p_3 \vp}\left(\p_3\p_t v - \frac{\p_t \vp}{\p_3 \vp}\p_3^2 v\right) + \left(\frac{-\p_3 \p_t\vp}{(\p_3 \vp)^2} + \frac{\p_t \vp}{\p_3 \vp}\frac{\p_3^2 \vp}{(\p_3 \vp)^2}\right)\p_3 v\\
  & + \frac{1}{\p_3 \vp}\overline{v}\cdot\left(\p_3\TP  v - \frac{\TP  \vp}{\p_3 \vp}\p_3^2 v\right) + \left(\frac{-\overline{v}\cdot\p_3 \TP \vp}{(\p_3 \vp)^2} + \frac{\overline{v}\cdot\TP  \vp}{\p_3 \vp}\frac{\p_3^2 \vp}{(\p_3 \vp)^2}\right)\p_3 v + \frac{v_3\p_3^2 v}{(\p_3 \vp)^2} - \frac{\p_3^2\vp}{(\p_3 \vp)^3}v_3\p_3 v.
\end{align*}
It follows that
\begin{align*}
  &\left(\int_{\Omega}\left|D_t^\vp\p_3^\vp v \TP^\alpha\vp\right|^2\p_3\vp\d x\right)^{1/2} \leq \|\p_3\vp\|_\infty^{1/2}\|\TP^\alpha\vp\|_\infty\|D_t^\vp\p_3^\vp v\|_0\\
  &\qquad \leq C\left(1 +  |\psi|_\infty\right)^{1/2}|\TP^\alpha \psi|_\infty\left\{\frac{1}{c_0}\left(\|\p_3\p_t v\|_0 + \frac{|\p_t\psi|_\infty}{c_0}\|\p_3^2 v\|_0\right)\right.\\
  &\qquad\qquad+ \frac{1}{c_0^2}\left(|\p_t\psi|_\infty + \frac{|\p_t\psi|_\infty}{c_0}\cdot |\psi|_\infty\right)\|\p_3 v\|_0 + \frac{\|\overline{v}\|_\infty}{c_0}\left(\|\p_3\TP  v\|_0 +\frac{|\TP \psi|_\infty}{c_0}\|\p_3^2 v\|_0\right)\\
  &\qquad\qquad\left. + \frac{1}{c_0^2}\left(|\TP \psi|_\infty + \frac{|\TP \psi|_\infty |\psi|_\infty}{c_0}\right)\|\overline{v}\|_\infty\|\p_3 v\|_0 + \frac{\|v_3\|_\infty}{c_0^2}\left(\|\p_3^2 v\|_0 + \frac{|\psi|_\infty}{c_0}\|\p_3 v\|_0\right)\right\}\\
  &\qquad \leq P\left(c_0^{-1},|\psi|_{C^3},|\psi_t|_{C^3}\right) \left\{\|\p_3\p_tv\|_0 + (1+\|v\|_\infty) \|v\|_2\right\}.
\end{align*}
Then invoke \eqref{eq-Dtvp} and write $\p_t v$ as
\begin{equation}\label{eq-pt-v}
  \p_t v = -\overline{v}\cdot \TP v - \frac{1}{\p_3 \vp}(v\cdot \N-\p_t\vp)\p_3 v - \ppk q.
\end{equation}
Consequently,
\begin{equation*}
  \|\p_3\p_tv\|_0 \leq P\left(c_0^{-1},|\psi|_{C^3},|\psi_t|_{C^1}\right) \Big(\left(1 + \|v\|_{W^{1,\infty}}\right)\|v\|_2 + \|\p q\|_1\Big).
\end{equation*}

Next, by using the commutator estimates \eqref{commu-es}, those term including $\widetilde{R^3}(v)$ (defined in \eqref{eq-R3-f}) can easily be bounded as
\begin{equation*}
  \left(\int_{\Omega}\left|\widetilde{R^3}(v)\right|^2\p_3\vp\d x\right)^{1/2} \leq P\left(c_0^{-1},|\psi|_{C^3}\right)\left(\|v\|_{W^{1,\infty}}\|v\|_3 + |\psi_t|_{C^2}\|\p v\|_2\right).
\end{equation*}
Combining the above estimates, we obtain
\begin{equation}\label{es-I3}
  I_3 \leq P\left(c_0^{-1},|\psi|_{C^3},|\psi_t|_{C^3}\right)\Big(\|\p q\|_1 + (1 + \|v\|_{W^{1,\infty}})\|v\|_3\Big)\left(\int_{\Omega}|\V|^2\p_3\vp\d x\right)^{1/2}.
\end{equation}

\noindent{\bf Control of $I_4$.} Recall the definitions \eqref{eq-tan-cmut-ppk} and \eqref{eq-R1-f} of remainders $R^2(\cdot)$ and $R^1(\cdot)$. We use the Cauchy-Schwarz inequality to get
\begin{align*}
  I_4 \leq \left(\int_{\Omega}|\V|^2\p_3\vp\d x\right)^{1/2}\left(\int_{\Omega}\left|\ppk_3\ppk q\TP^\alpha\vp + R^1(q)\right|^2\p_3\vp\d x\right)^{1/2}.
\end{align*}
We need to estimate the second factor above. Firstly, we expand $\p_3^\vp \ppk q$ as
\begin{align*}
  \p_3^\vp \p_i^\vp q = & \frac{\p_3\p_i q}{\p_3\vp} - \frac{\p_i\vp}{(\p_3\vp)^2}\p_3^2 q - \frac{(\p_3\p_i\vp)\p_3\vp - \p_i\vp\p_3^2\vp}{(\p_3\vp)^3}\p_3 q,\;i=1,2,\\
  \p_3^\vp \p_3^\vp q = & \frac{\p_3^2 q\p_3\vp - \p_3 q\p_3^2 \vp} {(\p_3\vp)^3}.
\end{align*}
Then we can obtain
\begin{align*}
  \|\p_3^\vp \ppk q\|_0 \leq & \frac{1}{c_0}\|\p_3\p_i q\|_0 + \frac{|\p_i\psi|_\infty}{c_0^2}\|\p_3^2 q\|_0 + \frac{|\p_i\psi|_\infty + |\p_i\psi|_\infty |\psi|_\infty}{c_0^3}\|\p_3 q\|_0  + \frac{1}{c_0^2}\|\p_3^2 q\|_0 + \frac{|\psi|_\infty}{c_0^3}\|\p_3 q\|_0\\
  \leq & P(c_0^{-1},|\psi|_{C^1})\|\p_3 q\|_{1}.
\end{align*}
Next, the remaining terms can easily be bounded as
\begin{equation*}
  \left(\int_{\Omega}\left|R^1(q)\right|^2\p_3\vp\d x\right)^{1/2} \leq P\left(c_0^{-1},|\psi|_{C^3}\right)\|\p_3 q\|_2.
\end{equation*}
Finally, it follows from the two estimates above that
\begin{equation}\label{es-I4}
  I_4 \leq P(c_0^{-1},|\psi|_{C^3})\|\p_3 q\|_2\left(\int_{\Omega}|\V|^2\p_3\vp\d x\right)^{1/2}.
\end{equation}

Since
$$E(t) = |\sqrt{\sigma}\psi|_4^2 + \|v\|_3^2,$$
we plug in the
 estimates \eqref{es-I1}, \eqref{es-I2}, \eqref{es-I3} and \eqref{es-I4} into \eqref{eq-AGU-V} to obtain
\blue{\begin{equation}\label{es-AGU-V-1}
  \begin{aligned}
  &\frac{1}{2}\ddt \int_{\Omega}|\V|^2\p_3\vp\d x + \frac{1}{2}\ddt\int_{\Gamma_\t}\tfrac{\left|\sqrt{\sigma}\TP^\alpha\TP\psi \right|^2}{(1+|\TP\psi|^2)^{1/2}} -\tfrac{\left|\sqrt{\sigma}\TP\psi\cdot \TP^\alpha\,\TP\psi\right|^2}{(1+|\TP\psi|^2)^{3/2}}\d x'\\
  \leq & P\left(c_0^{-1},|\psi(t)|_{C^3},|\psi_t(t)|_{C^3}\right) \left[\left(1 + |\vb|_{W^{1,\infty}}\right)E(t) + (1+|\vb|_{\infty})\|\p q\|_2\sqrt{E(t)} \,\right]\\
  &\quad + P\left(c_0^{-1},|\psi(t)|_{C^3},|\psi_t(t)|_{C^3}\right) \Big[\,\|\partial q\|_2 + (1 + \|v\|_{W^{1,\infty}})\sqrt{E(t)}\,\Big]\left(\int_{\Omega}|\V|^2\p_3\vp\d x\right)^{1/2}.
  \end{aligned}
\end{equation}}
\blue{On the other hand,  if $\vb(t)$ and $\TP  \vb(t)=\left(\p_1 \vb(t), \p_2\vb (t)\right)^T$ are continuous on ${\Omega}$, we plug in the
 estimates \eqref{es-I1 smh}, \eqref{es-I2}, \eqref{es-I3} and \eqref{es-I4} into \eqref{eq-AGU-V} to obtain
 \begin{equation}\label{es-AGU-V-1 smh}
  \begin{aligned}
  &\frac{1}{2}\ddt \int_{\Omega}|\V|^2\p_3\vp\d x + \frac{1}{2}\ddt\int_{\Gamma_\t}\tfrac{\left|\sqrt{\sigma}\TP^\alpha\TP\psi \right|^2}{(1+|\TP\psi|^2)^{1/2}} -\tfrac{\left|\sqrt{\sigma}\TP\psi\cdot \TP^\alpha\,\TP\psi\right|^2}{(1+|\TP\psi|^2)^{3/2}}\d x'\\
  \leq & P\left(c_0^{-1},|\psi(t)|_{C^3},|\psi_t(t)|_{C^3}\right) \left[\left(1 + \|v\|_{W^{1,\infty}}\right)E(t) + (1+|\vb|_{\infty})\|\partial q\|_2\sqrt{E(t)}\right] \\
  &\quad + P\left(c_0^{-1},|\psi(t)|_{C^3},|\psi_t(t)|_{C^3}\right) \Big[\,\|\partial q\|_2 + (1 + \|v\|_{W^{1,\infty}})\sqrt{E(t)}\,\Big]\left(\int_{\Omega}|\V|^2\p_3\vp\d x\right)^{1/2}.
  \end{aligned}
\end{equation}
}
Furthermore, noting that
\begin{equation*}
  \frac{\left|\sqrt{\sigma}\TP^\alpha\TP\psi \right|^2}{(1+|\TP\psi|^2)^{1/2}} -\frac{\left|\sqrt{\sigma}\TP\psi\cdot \TP^\alpha\,\TP\psi\right|^2}{(1+|\TP\psi|^2)^{3/2}} \geq \frac{|\sqrt{\sigma}\TP^\alpha\,\TP\psi|^2}{(1+|\TP\psi|^2)^{3/2}},
\end{equation*}
we  integrate \eqref{es-AGU-V-1} over $[0,t]$ where $ t\leq T^*$ to acquire \eqref{es-AGU-V-2}. This concludes the proof of Theorem \ref{thm-tangential}.
%\end{proof}
%=======================================================================================
\subsection{Elliptic Estimates for $q$}\label{subsect. q}
In light of Theorem \ref{thm-tangential}, we still require the control of $\|\p q\|_2$ to close the energy estimate. This is done by studying the elliptic equation verified by $q$:
\begin{align}\label{q elleq}
-\triangle^\vp q:=\pp\cdot(\pp q) = (\ppk v)^{T}: (\ppk v), \quad \text{in}\,\,\Omega,
\end{align}
which is derived by taking the divergence operator $\pp\cdot$ to the momentum equation
$$
\Dtp v+\pp q=0.
$$
In particular, we estimate $\|\p q\|_2$ by studying \eqref{q elleq} equipped with Neumann boundary condition on $\Gamma_{\t}$ (i.e., \eqref{eqs-q-Neu-BC}) using the Hodge-type elliptic estimate \eqref{es-Hodge-2}. In this process, however, we pick up a lower order quantity $\|\pp q\|_0$ which also has to be controlled.

\subsubsection{Estimate for $\|\pp q\|_0$}
We first bound $\|\pp q\|_0$ by considering the elliptic equation verified by $q$ equipped with the Dirichlet boundary condition:
\begin{equation}\label{eqs-q-Dir-BC}\left\{
  \begin{aligned}
  -\triangle^\vp q = & (\ppk v)^{T}: (\ppk v),\quad& \i \;\Omega,\\
  q = & -\sigma\TP \cdot \left(\tfrac{\TP \psi}{\sqrt{1+|\TP \psi|^2}}\right), \quad&\text{on}\;\Gamma_\t,\\
  n\cdot \p q = & 0, \quad&\text{on}\;\Gamma_\b.
  \end{aligned}\right.
\end{equation}
\begin{lemma}\label{lem q lower est}
If $q$ verifies \eqref{eqs-q-Dir-BC}, then \blue{for each sufficiently small $\epsilon>0$,} we have
\begin{align}\label{q lower est}
\|\pp q\|_0^2 \leq \epsilon P(|\Omega|, c_0^{-1}, |\psi|_{C^1})\|\p q\|_1^2+C(\epsilon^{-1})P(c_0^{-1}, |\psi|_{C^1})\left(\|\p v\|_{\infty}^2\|\p v\|_0^2+|\sqrt{\sigma}\psi|_4^2\right).
\end{align}
\end{lemma}
\begin{proof}
Invoking Lemma \ref{lem-A2}, we have
\begin{align*}
\int_{\Omega} |\pp q|^2\p_3\vp\d x = -\int_{\Omega} q(\triangle^\vp q)\p_3 \vp \d x+\int_{\Gamma_{\t}} q (N\cdot \pp q)\d x',
\end{align*}
where
\begin{align*}
-\int_{\Omega} q(\triangle^\vp q)\p_3\vp \d x \leq P(|\psi|_{C^1})\|\triangle^\vp q\|_0\|q\|_0
\leq P(|\psi|_{C^1})\Big(\epsilon \|q\|_0^2 + C(\epsilon^{-1})\|\triangle^\vp q\|_0^2\Big)\\
\leq \epsilon P(|\psi|_{C^1})\|q\|_0^2+C(\epsilon^{-1})P(c_0^{-1}, |\psi|_{C^1})\|\p v\|_{\infty}^2\|\p v\|_0^2,
\end{align*}
and
\begin{align*}
\int_{\Gamma_{\t}} q (N\cdot \pp q)\d x' \leq P(|\psi|_{C^1})\Big(\epsilon\|\pp q\|_1^2+C(\epsilon^{-1})|q|_0^2\Big)\\
\leq \epsilon P(c_0^{-1}, |\psi|_{C^1}) \|\p q\|_1^2+C(\epsilon^{-1})P(|\psi|_{C^1})\sum_{k=0}^1|\TP^{k+1}\psi|_0^2.
\end{align*}
Summing these up, we obtain
\begin{align}\label{est pp q}
\|\pp q\|_0^2 \leq P(c_0^{-1}, |\psi|_{C^1})\left(\epsilon \|q\|_0^2+\epsilon \|\p q\|_1^2\right)+C(\epsilon^{-1})P(c_0^{-1}, |\psi|_{C^1})\left(\|\p v\|_{\infty}^2\|\p v\|_0^2+|\sqrt{\sigma}\psi|_4^2\right).
\end{align}
On the other hand, using Poincar\'e's inequality, we get
\[
\|q\|_0^2 \leq C(|\Omega|)\Big(\|\p q\|_0^2+ \left(\int_{\Omega} q\d x\right)^2\Big).
\]
Let $X=(x^1,0,0)^T$. Then
\begin{align*}
\left(\int_{\Omega} q\d x\right)^2 = \left(\int_{\Omega} \p_i X^i q \d x\right)^2 = \left( \int_{\Omega} x^1\p_1 q\d x\right)^2 \leq C\|\p q\|_0^2.
\end{align*}
Thus,
\begin{align}\label{est q}
\|q\|_0^2 \leq C(|\Omega|)\|\p q\|_0^2.
\end{align}
Finally, \eqref{q lower est} follows from \eqref{est pp q} and \eqref{est q}.
\end{proof}
%===============================
\subsubsection{Estimate for $\|\p q\|_2$}
We next bound $\|\p q\|_2$ by considering the elliptic equation of $q$ equipped with Neumann boundary conditions.
To achieve this, we take the dot product of the momentum equation
with $\N$ to get:
\begin{equation*}
  (\p_t v)\cdot \N + (\overline{v}\cdot \TP v)\cdot \N  + (\frac{1}{\p_3 \vp}(v\cdot \N-\p_t\vp)\p_3 v )\cdot \N + \ppk q \cdot \N = 0.
\end{equation*}
Since
\begin{equation*}
  \begin{aligned}
  (\p_t v)\cdot \N\big|_{\Gamma_\t} = & \p_{t}^2\psi + (\ol{v}\cdot\TP) \p_t \psi,\\
  (\overline{v}\cdot \TP v)\cdot \N \big|_{\Gamma_\t} = & (\overline{v}\cdot \TP v)\cdot N,\\
  (v\cdot \N-\p_t\vp)\p_3^\vp v \cdot \N \big|_{\Gamma_\t} = & 0,
  \end{aligned}
\end{equation*}
we obtain
\begin{align*}
\N\cdot \ppk q \big|_{\Gamma_\t} = -(\overline{v}\cdot\TP v)\cdot N - \p_t^2 \psi - (\vb\cdot\TP)(v\cdot N),
\end{align*}
and
\begin{equation}\label{eqs-q-Neu-BC}\left\{
  \begin{aligned}
  -\triangle^\vp q = & (\ppk v)^{T}: (\ppk v),\quad& \i \;\Omega,\\
  N\cdot \ppk q = & -(\overline{v}\cdot\TP v)\cdot N - \p_t^2 \psi - (\overline{v}\cdot\TP)(v\cdot N), \quad&\text{on}\;\Gamma_\t,\\
  n\cdot \p q = & 0, \quad&\text{on}\;\Gamma_\b.
  \end{aligned}\right.
\end{equation}
\begin{remark}
We employ the Neumann boundary condition on $\Gamma_{\t}$ instead of the Dirichlet condition as it yields a regularity loss while estimating $q$ at the top order. Particularly, in light of \eqref{es-Hodge-1}, we require $\psi\in H^{4.5}(\Gamma_{\t})$ to control the mean curvature.
\end{remark}

To control $\|\p q\|_2^2$, it suffices to estimate $\|\ppk q\|_2^2$ thanks to the fact that
\begin{equation}\label{es-p-q-1}
    \|\p q\|_2^2 \leq P(c_0^{-1},|\psi|_{C^3})\|\ppk q\|_2^2.
  \end{equation}
Now, by applying \eqref{es-Hodge-2} to $\ppk q$ with $s=2$, we get
\begin{equation}\label{es-ppk-q}
  \begin{aligned}
  \|\ppk q\|_2^2 \leq & {P(|\psi|_{C^3})}\Big(\|(\ppk v)^{T}: (\ppk v)\|_1^2 + |(\overline{v}\cdot\TP v)\cdot N|_{1.5}^2 + |\p_t^2 \psi|_{1.5}^2 + |\overline{v}\cdot\TP(v\cdot N)|_{1.5}^2+\|\pp q\|_0^2\Big)\\
  \leq & {P(|\psi|_{C^3})} \left(\|v\|_{W^{1,\infty}}^2\|v\|_2^2 + \|v\|_\infty^2\|v\|_3^2 + |\p_t^2 \psi|_{1.5}^2+\|\pp q\|_0^2\right)\\
  \leq & {P(|\psi|_{C^3})} \left(\|v\|_{W^{1,\infty}}^2\|v\|_3^2 + |\p_t^2 \psi|_{1.5}^2+\|\pp q\|_0^2\right).
  \end{aligned}
\end{equation}
Finally, since \eqref{assump Kt} implies $|\psi|_{C^3}\leq M$, invoking \eqref{q lower est} and taking $\epsilon=\epsilon(M)$ sufficiently small, we infer from \eqref{es-ppk-q} and \eqref{es-p-q-1} that
\begin{align}\label{es-p-q-H2}
\begin{aligned}
\|\p q\|_2 \leq P(c_0^{-1},|\psi|_{C^3})\left((\|v\|_{W^{1,\infty}}+1)\sqrt{E(t)} + |\p_t^2 \psi|_{1.5}\right).
\end{aligned}
\end{align}

%=========================================================================================================

%=========================================================================================================
\subsection{Proof of Theorem \ref{thm-weak-CN-middle}}\label{sect. proof thm 1.1}
 Because Poincar\'{e}'s inequality implies that
\begin{equation}\label{ineq-Poincare}
  \frac{1}{C}|\psi|_4 \leq |\TP^4 \psi|_0 \leq C|\psi|_4,
\end{equation}
holds for some $C>0$,
 we deduce from \eqref{es-div-curl-ctrl-v} and \eqref{ineq-tan-es-f} that for any $t\in (0, T^*)$,
\begin{equation} \label{E control step 1}
  \begin{aligned}
    E(t) \leq & P(|\psi|_{C^3})\left(\|v\|_0^2+\|\omega^\vp\|_2^2 + \|\overline{\partial}^3 v\|_0^2 + |\sqrt{\sigma}\TP^4\psi(t)|_0^2\right)\\
    \leq & P(c_0^{-1}, |\psi|_{C^3})\left(\|v\|_0^2+\|\omega^\vp\|_2^2 + \|\V\|_0^2 + \|\p v\|_0^2+ |\sqrt{\sigma}\TP^4\psi(t)|_0^2\right),
  \end{aligned}
\end{equation}
where the second inequality follows from
$$
\|\ppk_3 v\TP^\alpha\vp\|_0^2 \leq P(c_0^{-1}, |\psi|_{C^3})\|\p v\|_0^2,\quad |\alpha|=3.
$$

Now, in view of \eqref{assump Kt}, it holds that
\begin{align}
P(c_0^{-1}, |\psi(t)|_{C^3}, |\psi_t(t)|_{C^3}) \leq P(c_0^{-1}, M),\quad \forall t\in (0, T^*).
\end{align}
Then, invoking the estimate of $\|v\|_0^2$ in \eqref{v in L2},  the estimate of $\|\omp\|_2^2$ in Lemma \ref{lem-omega-p-v}, the estimate of  $\|\p v\|_0^2$ in Lemma \ref{lem-p-v}, as well as the tangential estimate in Theorem \ref{thm-tangential}, we have, by \eqref{E control step 1}, that
\blue{\begin{equation}\label{Et1}
  \begin{aligned}
  E(t)\leq & P(c_0^{-1}, M)E(0) + P(c_0^{-1}, M)\int_{0}^{t}(1+\|v(\tau)\|_{W^{1,\infty}}) E(\tau) \d \tau\\
  & + P(c_0^{-1}, M)\int_{0}^{t}\left((1 + \|v(\tau)\|_{W^{1,\infty}})E(\tau) + \|\p q(\tau)\|_1\sqrt{E(\tau)}\right)\d \tau \\
  & +P(c_0^{-1}, M)\int_{0}^{t}\left(1 + |\vb(\tau)|_{W^{1,\infty}}\right)\Big(E(\tau) + \sqrt{E(\tau)}\Big) + (1+|\vb(\tau)|_{\infty})\|\partial q(\tau)\|_2\sqrt{E(\tau)}\d \tau\\
  &+P(c_0^{-1}, M)\int_{0}^{t}\Big((1 + \|v(\tau)\|_{W^{1,\infty}})E(\tau) + \|\partial q(\tau)\|_2\sqrt{E(\tau)}\Big)\d \tau.
  \end{aligned}
\end{equation}}
\blue{On the other hand, if $\vb(t)$ and $\TP  \vb(t)=\left(\p_1 \vb(t), \p_2\vb (t)\right)^T$ are continuous on ${\Omega}$, \eqref{Et1} becomes
\begin{equation}\label{Et1 smh}
  \begin{aligned}
  E(t)\leq & P(c_0^{-1}, M)E(0) + P(c_0^{-1}, M)\int_{0}^{t}(1+\|v(\tau)\|_{W^{1,\infty}}) E(\tau) \d \tau\\
  & + P(c_0^{-1}, M)\int_{0}^{t}\left((1 + \|v(\tau)\|_{W^{1,\infty}})E(\tau) + \|\p q(\tau)\|_1\sqrt{E(\tau)}\right)\d \tau \\
  & +P(c_0^{-1}, M)\int_{0}^{t}\left(1 + \|v(\tau)\|_{W^{1,\infty}}\right)\Big(E(\tau) + \sqrt{E(\tau)}\Big) + (1+|\vb(\tau)|_{\infty})\|\partial q(\tau)\|_2\sqrt{E(\tau)}\d \tau\\
  &+P(c_0^{-1}, M)\int_{0}^{t}\Big((1 + \|v(\tau)\|_{W^{1,\infty}})E(\tau) + \|\partial q(\tau)\|_2\sqrt{E(\tau)}\Big)\d \tau.
  \end{aligned}
\end{equation}
}
\blue{Since \eqref{assump Kt} implies also
\begin{align}\label{bound}
|\psi_{tt}(t)|_{1.5}+|\vb(t)|_{\infty}\leq M, \quad \forall t\in (0, T^*),
\end{align}
we invoke \eqref{es-p-q-H2} and then infer from \eqref{Et1} and \eqref{Et1 smh} that
\begin{equation}\label{Eest sob}
  \begin{aligned}
  E(t) \leq  P(c_0^{-1}, M)E(0) + P(c_0^{-1}, M)\int_{0}^{t}[\,1+ \|v(\tau)\|_{W^{1,\infty}}\,](\,E(\tau) + \sqrt{E(\tau)}\,)\d \tau\\
  + P(c_0^{-1}, M)\int_{0}^{t}[\,1+ |\vb(\tau)|_{\dot{W}^{1,\infty}}\,](\,E(\tau) + \sqrt{E(\tau)}\,)\d \tau,
  \end{aligned}
\end{equation}
and
\begin{equation}\label{Eest smh}
  \begin{aligned}
  E(t) \leq  P(c_0^{-1}, M)E(0) + P(c_0^{-1}, M)\int_{0}^{t}[\,1+ \|v(\tau)\|_{W^{1,\infty}}\,](\,E(\tau) + \sqrt{E(\tau)}\,)\d \tau,
  \end{aligned}
\end{equation}
respectively. }

\blue{In light of the standard inequality $\sqrt{E} \lesssim 1+ E$, we arrive at
\begin{equation}\label{E sob}
\begin{aligned}
  E(t)+\sqrt{E(t)} \leq P(c_0^{-1}, M)E(0) + P(c_0^{-1}, M)\int_{0}^{t}[\,1+ \|v(\tau)\|_{W^{1,\infty}}\,](\,E(\tau) + \sqrt{E(\tau)}\,)\d \tau\\
  + P(c_0^{-1}, M)\int_{0}^{t}[\,1+ |\vb(\tau)|_{\dot{W}^{1,\infty}}\,](\,E(\tau) + \sqrt{E(\tau)}\,)\d \tau,
  \end{aligned}
\end{equation}
from \eqref{Eest sob}, 
and
\begin{equation}\label{E smh}
  E(t)+\sqrt{E(t)} \leq P(c_0^{-1}, M)E(0) + P(c_0^{-1}, M)\int_{0}^{t}[\,1+ \|v(\tau)\|_{W^{1,\infty}}\,](\,E(\tau) + \sqrt{E(\tau)}\,)\d \tau,
\end{equation}
from \eqref{Eest smh}. 
Lastly, since $\vbi\leq M$, $\forall t\in (0, T^*)$, Gr\"{o}nwall's inequality implies that from either \eqref{E sob} or \eqref{E smh}, we have
\begin{equation}\label{Et final}
  \begin{aligned}
  E(t)+\sqrt{E(t)} \leq & C(c_0^{-1}, M, E(0))\exp \Big( \int_0^t C(c_0^{-1}, M)(1+\|v(\tau)\|_{W^{1,\infty}})\,d\tau\Big).
  \end{aligned}
\end{equation}
holds $\forall t\in (0, T^*)$.}
This concludes the proof of Theorem \ref{thm-weak-CN-middle}.
%=========================================================================================================
\section{Proof of Theorem \ref{thm-strong-CN} }\label{sect. proof thm 1.2}
Parallel to the proof of Theorem \ref{thm-weak-CN}, i.e.,  we assume $T^*<+\infty$, and none of the conditions (a), (b'), and (c) hold in Theorem \ref{thm-strong-CN}. Our goal is to show:
\begin{theorem}\label{thm-strong-CN-middle}
Suppose $T^*<+\infty$, and
 there exist constants $M, c_0>0$, such that
\begin{align}
\sup_{t\in [0, T^*)} \K(t) \leq M, \\
\inf_{t\in [0,T^*)}\p_3 \vp(t) \geq c_0, \\
\inf_{t\in [0, T^*)} (b-|\psi(t)|_{\infty}) \geq c_0.
\end{align}
Then $\|v(t)\|_3$, $t\in [0, T^*)$ is bounded whenever the quantity $\int_0^t \|\omp(\tau)\|_{\infty}\d t$ remains finite.
\end{theorem}
\subsection{Two Key Lemmas}
Particularly, since $|\psi(t)|_{C^3} \leq M$ for all $t\in [0, T^*)$, the results of \cite{Ferrari-1993CMP} suggest:
\begin{lemma}\label{thm Ferrari}
Let $U=U(t,y)$ be a smooth vector field defined on $\DD_t$, satisfying
\begin{equation}\label{equ U}
\begin{aligned}
\nab \cdot U=0, \quad& \text{in}\,\,\DD_t,\\
U\cdot N = 0,\quad& \text{on}\,\,\p\DD_{t,\t},\\
U\cdot n=0,\quad &\text{on}\,\,\p\DD_{t,\b}.
\end{aligned}
\end{equation}
Then
\begin{align}\label{ferrari}
\|U(t)\|_{W^{1,\infty}(\DD_t)} \leq C\left((1+\log^{+}\|\nab\times U(t)\|_{H^2(\DD_t)})\|\nab\times U(t)\|_{L^\infty(\DD_t)}+1\right),
\end{align}
holds for all $t\in [0, T^*)$.
Here,  $\log^+ f = \log f$ if $f\geq 1$, $\log^+ f=0$ otherwise.
\end{lemma}
\begin{proof}
The estimate \eqref{ferrari} follows from
\cite[Proposition 1 and Corollary 1]{Ferrari-1993CMP}, thanks to the fact that
\begin{equation}\label{fact}
\p \DD_{t,\t} \in C^3,\quad \p\DD_{t,\t}\cap \p\DD_{t,\b}=\emptyset, \quad \forall t\in [0, T^*).
\end{equation}
Here,  $\p \DD_{t,\t} \in C^3$ follows from $|\psi(t)|_{C^3}\leq M$, whereas $\p\DD_{t,\t}\cap \p\DD_{t,\b}=\emptyset$ is a direct consequence of $b-|\psi(t)|_{\infty}\geq c_0$.
 \end{proof}
 Furthermore, the following Schauder-type estimate also holds on $\DD_t$:
 \begin{lemma} \label{thm Nardi} Let $\xi$ be a smooth function defined on $\DD_t$ satisfying the boundary value problem:
 \begin{align}\label{xi}
\begin{aligned}
\triangle \xi = 0,\quad&\text{in}\,\,\,\DD_t,\\
N\cdot \nab \xi=\beta,\quad&\text{on}\,\,\,\p\DD_{t,\t},\\
n\cdot  \nab\xi = 0,\quad&\text{on}\,\,\,\p\DD_{t,\b},
\end{aligned}
 \end{align}
 where $\beta:\p\DD_t\to \R$ is a given smooth function. Then it holds that
\begin{equation}\label{nardi}
\|\xi\|_{C^{2,\gamma}(\DD_t)} \leq C|\beta|_{C^{1,\gamma}(\p\DD_{t,\t})}, \quad 0<\gamma<1.
\end{equation}
 \end{lemma}
 \begin{proof}
Similar to the proof of Lemma \ref{thm Ferrari}, \eqref{nardi} follows from \cite[Theorem 4]{Nardi} and \eqref{fact}.
 \end{proof}

\subsection{The Eulerian Sobolev and H\"older Norms}
We prove in this subsection that the Eulerian Sobolev and H\"older norms can be transformed to the associated norms in the flat coordinates characterized by the diffeomorphism $\Phi(t,\cdot)$, as long as $\psi(t)\in C^3$.

The Eulerian Sobolev norm $\|\cdot\|_{H^s(\DD_t)}$ is defined via the Eulerian spatial derivatives $\nab_i=\pp_i$, $i=1,2,3$, defined in \eqref{op-diff-fix}.  Let $f:\DD_t\to \R$ be a generic smooth function.
We can see that
\begin{equation}\label{Eulerian Sobolev}
\|f\|_{H^s(\DD_t)} \leq P(c_0^{-1}, |\psi|_{C^3})\|f\circ\Phi(t,\cdot)\|_{s}\leq P(c_0^{-1},M)\|f\circ\Phi(t,\cdot)\|_{s},\quad s\leq 3.
\end{equation}
Similarly, there exists a constant $C=C(c_0^{-1}, M)>0$, such that
\begin{align}\label{comparable W1infty}
C^{-1}\|f\|_{W^{1,\infty}(\DD_t)}\leq \|f\circ \Phi(t,\cdot)\|_{W^{1,\infty}} \leq C\|f\|_{W^{1,\infty}(\DD_t)}.
\end{align}
In other words, $\|f\|_{W^{1,\infty}(\DD_t)}$ and $\|f\circ \Phi(t,\cdot)\|_{W^{1,\infty}}$ are comparable with each other.
Furthermore, the Eulerian H\"older norm $|\cdot|_{C^{1,\gamma}(\p\DD_{t,\t})}$ is defined through the Eulerian tangential spatial derivatives:
\begin{equation*}
\overline{\pp_i}:= \left(\pp_i - \frac{\N\cdot \pp}{|\N|^2}\N_i\right)\bigg|_{\p\DD_{t,\t}},\quad \text{with}\,\,\,i=1,2,3.
\end{equation*}
By a direct calculation, we obtain
\begin{equation*}
\frac{\N\cdot \pp}{|\N|^2} = -\frac{\p_1\vp\p_1+\p_2\vp\p_2}{1+|\TP\vp|^2}+\frac{1}{\p_3\vp}\p_3.
\end{equation*}
Thus, for $j=1,2$,
\begin{align*}
\overline{\pp_j}=\p_j-\frac{(\p_j\psi)(\p_1\psi\p_1+\p_2\psi\p_2)}{1+|\TP\psi|^2},
\end{align*}
as well as
\begin{align*}
\overline{\pp_3} = \frac{\p_1\psi\p_1+\p_2\psi\p_2}{1+|\TP\psi|^2}.
\end{align*}
Therefore, we conclude:
\begin{align}\label{Eulerian holder}
|\beta|_{C^{1,\gamma}(\p\DD_{t,\t})} \leq P(c_0^{-1}, |\psi|_{C^3}) |\beta\circ \Phi|_{C^{1,\gamma}}\leq P(c_0^{-1}, M)|\beta\circ \Phi(t,\cdot)|_{C^{1,\gamma}}.
\end{align}
\subsection{The Modified Velocity Field}
Let $\xi$ be defined by \eqref{xi} with $\beta=u\cdot N$.
We set
$$
\tu:=\nab \xi,
$$
and let
\begin{align}\label{modified velocity}
V=u-\tu
\end{align}
to be the modified velocity field. The construction of $U$ indicates that
\begin{align}\label{property U}
\begin{aligned}
\nab\cdot V=0, \quad \nab\times V=\omega\,(:=\nab\times u),\quad&\text{in}\,\,\DD_t,\\
V\cdot N=0,\quad V\cdot n=0,\quad& \text{on}\,\,\p\DD_{t,\t}\cup \p\DD_{t,\b}.
\end{aligned}
\end{align}
\subsection{Proof of Theorem \ref{thm-strong-CN-middle}}
We now invoke Lemma \ref{thm Ferrari} to obtain:
\begin{align}
\|V\|_{W^{1,\infty}(\DD_t)} \lesssim (1+\log^{+}\|\omega\|_{H^2(\DD_t)})\|\omega\|_{L^\infty(\DD_t)}+1
\lesssim \log(e+\|u\|_{H^3(\DD_t)})\|\omega\|_{L^\infty(\DD_t)}+1,
\end{align}
which implies
\begin{align}\label{U step 1}
\|u\|_{W^{1,\infty}(\DD_t)} \lesssim\log(e+\|u\|_{H^3(\DD_t)})\|\omega\|_{L^\infty(\DD_t)}+\|\tu\|_{W^{1,\infty}(\DD_t)}+1.
\end{align}
Here, in light of Lemma \ref{thm Nardi} and \eqref{Eulerian holder}, we have
\begin{align}
\begin{aligned}
\|\tu\|_{W^{1,\infty}(\DD_t)} &\leq \|\xi\|_{C^{2,\gamma}(\DD_t)}\leq C|u\cdot N|_{C^{1,\gamma}(\DD_t)}\\
&\leq P(c_0^{-1}, M)|v\cdot N|_{C^{1,\gamma}}=P(c_0^{-1},M)|\psi_t|_{C^{1,\gamma}} \leq C(c_0^{-1},M).
\end{aligned}
\end{align}
Now, thanks to \eqref{Eulerian Sobolev} and \eqref{comparable W1infty}, we deduce from \eqref{U step 1} that
\begin{align}\label{U step 2}
\|v\|_{W^{1,\infty}} \leq C(c_0^{-1},M)\left(\log\Big(e+C(c_0^{-1},M)\|v\|_{3}\Big)\|\omp\|_{\infty}+1\right).
\end{align}

Let $\mathcal{C}$ be a generic positive constant depends on $c_0^{-1}, M$, and $E(0)$.
Because \eqref{Et final} implies
\begin{equation}
  \begin{aligned}
 \|v(t)\|_3 \leq \mathcal{C}\exp \Big( \int_0^t \mathcal{C}(1+\|v(\tau)\|_{W^{1,\infty}})\,d\tau\Big),
  \end{aligned}
\end{equation}
we plug \eqref{U step 2} into the RHS and get
\begin{equation}\label{U step 3}
\begin{aligned}
\|v(t)\|_3 \leq \mathcal{C}\exp \Big( \int_0^t \mathcal{C} \left(1+\log(e+\mathcal{C}\|v(\tau)\|_3)\|\omp(\tau)\|_{\infty}\right) \,d\tau\Big).
\end{aligned}
\end{equation}
Let
$$
\mathcal{F}(t):= e+\mathcal{C}\|v(t)\|_3.
$$
Then we infer from \eqref{U step 3} that
\begin{equation}\label{F step 1}
\begin{aligned}
\log\mathcal{F}(t) \leq \log^+\mathcal{C}+\int_0^t \mathcal{C} \left(1+\|\omp(\tau)\|_{\infty}\log \mathcal{F}(\tau)\right) \,d\tau,
\end{aligned}
\end{equation}
which concludes the proof of the theorem.

\section{Remarks on Recovering the Regularity Loss in Theorems \ref{thm-weak-CN} and \ref{thm-strong-CN} with Modified Control Norms}\label{sect. lossless}
In Theorems \ref{thm-weak-CN} and \ref{thm-strong-CN}, we require the solution $(v(t), \psi(t))\in H^s(\Omega)\times H^{s+1}(\Gamma_{\t})$, $s>\frac{9}{2}$, but the space of continuation is merely $H^3(\Omega)\times H^4(\Gamma_{\t})$. This regularity loss is caused by the control norms in $\K(t)$, in particular, $|\psi_t|_{C^3} = |v\cdot N|_{C^3}$ cannot be controlled by $E(t)$.  Nevertheless, the double linear estimate \eqref{strategy energy} remains valid by replacing $\K(t)$ by
\blue{\begin{equation}\label{KM}
\begin{aligned}
\KM(t) &:= \KM_1(t)+\K_2(t),\\
\KM_1(t) &= |\psi(t)|_{C^3} + |\psi_t(t)|_{C^2}+|\psi_t(t)|_3 + |\psi_{tt}(t)|_{1.5}. 
\end{aligned}
\end{equation}}
In other words, we replace $|\psi_t|_{C^3}$ in $\K(t)$ by $|\psi_t|_{C^2}+|\psi_t|_3$. It is straightforward to see that both $|\psi_t|_{C^2}$ and $|\psi_t|_3$ reduces to $0$ if $\psi_t=v\cdot N=0$ on $\Gamma_{\t}$. This indicates that the reduction in Remark \ref{kt last} remains valid.

By repeating the analysis in Section \ref{sect. thm 1.1} with $\KM(t)$, we obtain
\blue{\begin{align}\label{strategy energy'}
 E(t)+\sqrt{E(t)}\leq P(c_0^{-1}, \KM_1(t))E(0) &+ \int_0^t P(c_0^{-1}, \KM_1(\tau))\left([\,1+ \|v(\tau)\|_{W^{1,\infty}}\,][\,1+ |\vb(\tau)|_{\infty}\,](\,E(\tau) + \sqrt{E(\tau)}\,)\right)\d \tau\nonumber\\
 &+ \int_0^t P(c_0^{-1}, \KM_1(\tau))\left([\,1+ |\vb(\tau)|_{W^{1,\infty}}\,](\,E(\tau) + \sqrt{E(\tau)}\,)\right)\d \tau,
\end{align}
where the second line  drops if $\vb(t)$, $\p_1 \vb(t)$, and $\p_2 \vb(t)$ are continuous on $\Omega$.}

Next, we prove that all quantities in $\KM(t)$ can be controlled by the energy that ties to the local existence, i.e.,
$$
E_{\text{exist}}(t)= \sum_{k=0}^{3}\left(\|\p_t^k v(t)\|_{3-k}^2+\sigma |\p_t^k\psi|_{4-k}^2\right).
$$
\begin{theorem}\label{thm-KM}
Let $E_{\text{exist}}(t)$ be defined as above.
For fixed $t\geq 0$ such that $E_{\text{exist}}(t)<+\infty$, it holds that
\begin{equation}\label{KM control}
\KM(t) \leq  P(E_{\text{exist}}(t)),
\end{equation}
provided that $\psi\in C^{1,\gamma}(\Gamma_{\t})$.
\end{theorem}
\begin{proof}
First, with the help of the standard Sobolev inequalities,  it is clear that $|\psi_t|_{3}$, $|\psi_{tt}|_{1.5}$, and $|\vb|_{W^{1,\infty}}$ are bounded by $E_{\text{exist}}$. Second, by rewriting the boundary condition of $q$ as
\begin{align}\label{q ell}
-\TP \cdot \left(\frac{\TP \psi}{|N|}\right) = \sigma^{-1} q,
\end{align}
and then applying the standard Schauder estimate, we have
\begin{align}\label{boost}
|\psi|_{C^{2,\gamma}}\leq C(\sigma^{-1}, \gamma, |\psi|_{C^{1,\gamma}})\left( |\psi|_{C^0}+|q|_{C^{0,\gamma}}\right).
\end{align}
Here, $|q|_{C^{0,\gamma}} \lesssim \|q\|_{1.5+\gamma+\delta}$ for some $\delta>0$, thanks to the standard Sobolev inequalities. In addition to this, when $\gamma+\delta<\frac{1}{2}$,  we infer from \cite[Proposition 3.1]{GuLuoZhang} (with $\mathbf{b}_0=0$ therein) that $\|q\|_{1.5+\gamma+\delta}\leq \|q\|_2 \leq P(E_{\text{exist}})$.

Third, by taking $\p_\tau$ with $\tau=1,2$ to \eqref{q ell}, we obtain
$$
-\TP \cdot \left( \frac{\TP \p_\tau \psi}{|N|} - \frac{\TP \psi\cdot \TP \p_\tau \psi}{|N|^3}\TP\psi\right) =\sigma^{-1} \p_\tau q.
$$
 Since \eqref{boost} implies $\psi\in C^{2,\gamma}(\Gamma_{\t})$,
 the standard Schauder estimate yields
\begin{align}
|\TP\psi|_{C^{2, \gamma}} \leq C(\sigma^{-1}, \gamma, |\psi|_{C^{2,\gamma}})\left(|\TP \psi|_{C^0} + |q|_{C^{1,\gamma}}\right),
\end{align}
where $|q|_{C^{1,\gamma}} \lesssim \|q\|_{2.5+\gamma+\delta}$ for some $\delta>0$, and $\|q\|_{2.5+\gamma+\delta}\leq \|q\|_3 \leq P(E_{\text{exist}})$ whenever $\gamma+\delta<\frac{1}{2}$.

Finally, we treat $|\psi_t|_{C^2}$ by a similar argument. Indeed, since $\psi_t$ verifies
$$
-\TP \cdot \left( \frac{\TP \psi_t}{|N|} - \frac{\TP \psi\cdot \TP\psi_t}{|N|^3}\TP\psi\right) =\sigma^{-1} q_t,
$$
and thus Schauder estimate yields
\begin{align}
|\psi_t|_{C^2} \leq |\psi_t|_{C^{2, \gamma}} \leq C(\sigma^{-1}, \gamma, |\psi|_{C^{2,\gamma}})\left( |\psi_t|_{C^0} + |q_t|_{C^{0,\gamma}}\right).
\end{align}
Again, we invoke the Sobolev inequalities and then \cite[Proposition 3.1]{GuLuoZhang} to control $|q_t|_{C^{0,\gamma}}$ by $\|q_t\|_{1.5+\gamma+\delta}\leq \|q_t\|_2\leq  P(E_{\text{exist}})$. Also, standard Sobolev inequalities imply $|\psi_t|_{C^0} = |v\cdot N|_{C^0} \leq P(E_{\text{exist}})$.
\end{proof}

Also, thanks to \eqref{strategy energy'}, we can adapt the arguments in Subsection \ref{sect. proof thm 1.1} and Section \ref{sect. proof thm 1.2} to show:
\begin{theorem}\label{thm-lossless}
 Let $(v(t),\psi(t))\in H^{3}(\Omega)\times H^{4}(\Gamma_{\t})$ be the solution of \eqref{eqs-incomp-Euler-fix}. Let
  \begin{align*}
    T^* = \sup\left\{T>0\,\big|\,(v(t),\psi(t)) \text{ can be continued in the class}\;C([0,T]; H^3(\Omega)\times H^{4} (\Gamma_{\t}))\right\}.
  \end{align*}
   If $T^*<+\infty$, then at least one of the following three  statements hold:
  \begin{itemize}
  \item [a'.] \begin{equation}
    \lim_{t\nearrow T^*}\KM(t) = + \infty,
  \end{equation}
  \item [b'.]   \begin{equation}
    \int_{0}^{T^*}\|\omp(t)\|_{L^\infty}\d t = +\infty,
  \end{equation}
  \item [c.]\begin{equation}
  \lim_{t\nearrow T^*} \left (\frac{1}{\p_3 \vp(t)} + \frac{1}{b-|\psi(t)|_{\infty}} \right)=+\infty,
  \end{equation}
  or turning occurs on the moving surface boundary.
  \end{itemize}
   \blue{Also, parallel to Theorem \ref{thm-strong-CN},  if $\vb(t)$, $\p_1 \vb(t)$, and $\p_2 \vb(t)$ are continuous on $\Omega$, then $\vbi$ in $\KM(t)$ can be dropped. }
\end{theorem}

We regard Theorem \ref{thm-lossless} as a generalized Beale-Kato-Majda-type breakdown criterion without regularity loss. Specifically, for $(v(t), \psi(t))\in H^3(\Omega)\times H^4(\Gamma_{\t})$, the control norms in $\KM(t)$ remains to be lossless as long as $E_{\text{exist}}(t)$ is finite.

\begin{appendix}
\section{The Reynold Transport Theorems}
\begin{lemma}\label{lem-A1}
  Let $f,\, g$ be smooth functions defined on $[0,T]\times \Omega$. Then there holds that
  \begin{equation}\label{eq-A1}
    \frac{d}{dt} \int_{\Omega} f g \p_3\vp \d x = \int_{\Omega} (\p_t^\vp f)g \p_3\vp \d x + \int_{\Omega} f(\p_t^\vp g) \p_3\vp \d x + \int_{\Gamma_\t}fg \p_t \psi \d x'.
  \end{equation}
\end{lemma}

\begin{proof}
  We exchange $\p_t$ and the integral in $\p_t \int_{\Omega} f g \p_3\vp \d x$ and use the definition \eqref{op-diff-fix} of $\p_t^\vp$ to get
  \begin{equation*}
    \begin{aligned}
    \frac{d}{dt} \int_{\Omega} f g \p_3\vp \d x = & \int_{\Omega} (\p_t f) g \p_3\vp \d x + \int_{\Omega} f(\p_t g) \p_3\vp \d x + \int_{\Omega} f g \p_t \p_3\vp \d x\\
    = & \int_{\Omega} (\p_t^\vp f) g \p_3\vp \d x + \int_{\Omega} f(\p_t^\vp g) \p_3\vp \d x + \int_{\Omega} f g \p_t \p_3\vp \d x\\
    & \qquad + \underbrace{\int_{\Omega} \p_t \vp \p_3 f g\d x}_{A} + \underbrace{\int_{\Omega} \p_t \vp f \p_3 g\d x}_{B}.
    \end{aligned}
  \end{equation*}
  Since $\p_t\vp|_{\Gamma_{\b}}=0$,
   we integrate $\p_3$ in B by parts to give us
  \begin{equation*}
    b = \int_{\Gamma_\t}fg \p_t \psi \d x' - \int_{\Omega} f g \p_t \p_3\vp \d x - A,
  \end{equation*}
  which concludes the proof of \eqref{eq-A1}.
\end{proof}
%==========================================================================================
\begin{lemma}\label{lem-A2}
  Let $f,\, g$ be defined as in Lemma \ref{lem-A1}. Then there holds that for $i=1,2$:
  \begin{equation}\label{eq-A2}
    \int_{\Omega}\left(\p_i^\vp f\right)g \p_3\vp\d x = -\int_{\Omega}f\left(\p_i^\vp g\right) \p_3\vp\d x + \int_{\Gamma_\t}fg N_i \d x'.
  \end{equation}
  Furthermore, if $g |_{\Gamma_\b}=0$, then
  \begin{equation}\label{eq-A2-1}
    \int_{\Omega}\left(\p_3^\vp f\right)g \p_3\vp\d x = -\int_{\Omega}f\left(\p_3^\vp g\right) \p_3\vp\d x + \int_{\Gamma_\t}fg N_3 \d x'.
  \end{equation}
\end{lemma}

\begin{proof}
  We consider the cases when $i=1,\,2$ and $i=3$ respectively.  We have
  \begin{equation*}
    \int_{\Omega}\left(\p_i^\vp f\right)g \p_3\vp\d x = \underbrace{\int_{\Omega}\p_i^\vp (fg) \p_3\vp\d x}_{C} - \int_{\Omega} f \left(\p_i^\vp g\right) \p_3\vp\d x.
  \end{equation*}
  Let $i=1,2$. Note that $\p_i \vp\big|_{\Gamma_\b}=0$ and $\p_i \vp\big|_{\Gamma_\t}=\p_i \psi$. We expand C as
  \begin{equation*}
    \begin{aligned}
    C = & \int_{\Omega}\p_i(fg) \p_3\vp\d x - \int_{\Omega}\p_3(fg)\p_i\vp\d x\\
    =& - \int_{\Omega}(fg) \p_i\p_3\vp\d x - \left\{\int_{\Gamma_\t}(fg) \p_i\vp\d x' - \int_{\Gamma_\b}(fg) \p_i\vp\d x' - \int_{\Omega}(fg) \p_3\p_i\vp\d x\right\}\\
    = & \int_{\Gamma_\t}fg N_i \d x'.
    \end{aligned}
  \end{equation*}
On the other hand, in the case when $i=3$, since $g|_{\Gamma_\b}=0$, we have
  \begin{equation*}
    \int_{\Omega}\left(\p_3^\vp f\right)g \p_3\vp\d x = \int_{\Omega} (\p_3 f) g\d x =  \int_{\Gamma_\t}fg \d x' -\int_{\Omega}f \left(\p_3^\vp g\right) \p_3\vp\d x.
  \end{equation*}
\end{proof}
%==========================================================================================
\begin{theorem}\label{thm-A3}
  Let $f$ be described as in Lemma \ref{lem-A1}. Then we have
  \begin{equation}\label{eq-A3}
    \frac{1}{2}\frac{d}{dt} \int_{\Omega}|f|^2 \p_3\vp\d x = \int_{\Omega}\left(D_t^\vp f\right) f\p_3\vp\d x.
  \end{equation}
\end{theorem}

\begin{proof}
  We expand the RHS of \eqref{eq-A3} to get
  \begin{equation*}
    \int_{\Omega}\left(D_t^\vp f\right) f\p_3\vp\d x = \underbrace{\int_{\Omega}\left(\p_t^\vp f\right) f\p_3\vp\d x}_{i} + \underbrace{\int_{\Omega}\left(v\cdot\ppk f\right) f\p_3\vp\d x}_{ii}.
  \end{equation*}
  Invoking Lemma \ref{lem-A1},
  \begin{equation*}
    \begin{aligned}
    i = & \frac{d}{dt}\int_{\Omega}|f|^2 \p_3\vp \d x - \int_{\Omega}f \left(\p_t^\vp f\right) \p_3\vp\d x - \int_{\Gamma_\t} |f|^2 \p_t \psi\d x' \\
    = & \frac{1}{2} \frac{d}{dt}\int_{\Omega}|f|^2 \p_3\vp \d x - \frac{1}{2} \int_{\Gamma_\t} |f|^2 \p_t \psi\d x'.
    \end{aligned}
  \end{equation*}
 Also, since $\ppk \cdot v = 0$ and $(v\cdot\N)\big|_{\Gamma_\b}=0$, Lemma \ref{lem-A2} indicates
  \begin{equation*}
    ii = \frac{1}{2}\int_{\Omega}\ppk\cdot(v|f|^2)\p_3\vp\d x - \frac{1}{2}\int_{\Omega}\left(\ppk \cdot v\right)|f|^2\p_3\vp\d x = \frac{1}{2}\int_{\Gamma_\t} |f|^2 v\cdot N \d x'.
  \end{equation*}
  Since $\p_t \psi = v\cdot N$ on $\Gamma_\t$, we complete the proof by summing up $i$ and $ii$.

\end{proof}
%==========================================================================================
\begin{corollary}\label{cor-A4}
  Let $f,\, g$ be defined as in Lemma \ref{lem-A1}. Then it holds that
  \begin{equation}\label{eq-A4}
    \frac{d}{dt} \int_{\Omega} f g \p_3\vp \d x = \int_{\Omega} \left(D_t^\vp f\right)  g \p_3\vp \d x + \int_{\Omega} f \left(D_t^\vp g\right) \p_3\vp \d x.
  \end{equation}
\end{corollary}

\begin{proof}
 We infer from \eqref{eq-A1} that
  \begin{equation*}
    \int_{\Omega} \left(\p_t^\vp f\right) g \p_3\vp \d x = \frac{d}{dt}\int_{\Omega} fg \p_3\vp \d x - \int_{\Omega} f\left(\p_t^\vp g\right) \p_3\vp \d x - \int_{\Gamma_\t} fg \p_t\psi\d x',
  \end{equation*}
  and \eqref{eq-A2}--\eqref{eq-A2-1} yield that
  \begin{equation*}
    \begin{aligned}
    \int_{\Omega} \left(v\cdot\ppk f\right) g \p_3\vp \d x = & \int_{\Omega} \ppk\cdot \left(v f\right) g \p_3\vp \d x - \int_{\Omega}  \left(\ppk\cdot v\right) fg \p_3\vp \d x\\
    = & \int_{\Gamma_\t} fg v\cdot N \d x' - \int_{\Omega} f \left(v\cdot\ppk g\right)\p_3\vp \d x.
    \end{aligned}
  \end{equation*}
  Then we get \eqref{eq-A4} by adding these identities up.
\end{proof}

\section{Calculus}
\begin{lemma} [\cite{Taylor}]\label{lem-B1}
  Let $s\geq 1$. There exists a constant $C > 0$ such that,
  \begin{itemize}
    \item[(1)] $\forall\,f,\,g\in H^s(\Omega)\cap C(\Omega)$, there holds
        \begin{equation}\label{Prod-Es}
          \|fg\|_s \leq C\big\{\|f\|_s\|g\|_\infty + \|f\|_\infty\|g\|_s\big\}.
        \end{equation}
    \item[(2)] If $f\in H^s(\Omega)\cap C^1(\Omega)$ and $g\in H^{s-1}(\Omega)\cap C(\Omega)$, then for $|\alpha|\leq s$,
    \begin{equation}\label{commu-es}
          \|[\p^\alpha,\,f]g\|_0 \leq C\big\{\|f\|_s\|g\|_\infty + \|f\|_{W^{1,\infty}(\Omega)}\|g\|_{s-1}\big\}.
        \end{equation}
  \end{itemize}
\end{lemma}

\blue{\begin{lemma} \label{thm-Linfty}
Let $f$ be a continuous function defined on $\Omega$.  Then
\begin{align}\label{Linfty ineq}
|f|_{\infty} \leq \|f\|_{\infty}.
\end{align}
\end{lemma}
\begin{proof}
We proceed with the proof by contradiction.
Let $\|f\|_{\infty} = L$. We assume that \eqref{Linfty ineq} is false, then there exists an $\epsilon>0$ such that
\begin{align} \label{Linfty false}
|f|_{\infty} \geq L+4\epsilon.
\end{align}
Since $\Gamma_{\t}$ is compact, and so there exists a point $\bar{x}\in \Gamma_{\t}$ such that $|f|_{\infty}=|f(\bar{x})|$. Let $\{x_n\}\subset \text{int }\Omega$ be the sequence that converges to $\bar{x}$, where $\text{int }\Omega$ is the interior of $\Omega$. Since $f$ is continuous on $\Omega$, there exists an $N>0$ such that $|f(x_N)-f(\bar{x})|<2\epsilon$. On the other hand, the continuity of $f$ also implies that there exists an $\delta>0$ such that $|f(x)-f(x_N)|<\epsilon$ holds for all $x\in B_{\delta}(x_N)$, i.e., the ball centered at $x_N$ with radius $\delta$. Thus, it holds that $|f(x)-f(\bar{x})|<3\epsilon$ for all $x\in B_{\delta}(x_N)$. Together with \eqref{Linfty false}, this implies that $\|f\|_{\infty}\geq L+\epsilon$, which contradicts the definition of $L$.
\end{proof}
\begin{remark}
This theorem is false if $f$ is merely continuous almost everywhere on $\Omega$. For instance, we consider $f:\Omega \to \R$ given by
\begin{align*}
f(x', x_3)= 
\begin{cases}
1,\quad |x'|<\frac{1}{2}, \,\,x_3=0,\\
0,\quad \text{otherwise}.
\end{cases}
\end{align*}
Then $f=0$ (and thus continuous) almost everywhere on $\Omega$, and thus $\|f\|_{\infty}=0$. However, $|f|_{\infty}=1$, which violates \eqref{Linfty ineq}. 
\end{remark}}

\section{The Hodge-type elliptic estimate}\label{appendix C}
\begin{theorem}
Let $X$ be a smooth vector field. Let $s\geq 1$ be an integer. Then
  \begin{equation}\label{es-Hodge-1'}
    \|X\|_s^2 \leq C_0(|\psi|_{C^s})\left(\|\ppk\cdot X\|_{s-1}^2 + \|\ppk\times X\|_{s-1}^2 + \|\overline{\partial}^s X\|_0^2+\|X\|_0^2\right).
  \end{equation}
\end{theorem}
\begin{proof}
This theorem is essentially Lemma B.2 of \cite{GLL}, whose proof is built on the following Hodge-type decomposition,
\begin{align}\label{Hodge-type}
|\pp X| \leq C(|\psi|_{C^s})\left(|\pp\cdot X|+|\pp\times X|+|\TP X|\right),
\end{align}
which is Lemma B.1 of \cite{GLL}. We recall that $\pp_i = \A_i^j\p_j$, and $\p_i = (\A^{-1})_i^j \pp_j$, where
\begin{equation*}
\A := \begin{pmatrix}
1&0&-\frac{\p_1 \varphi}{\p_3\varphi}\\
0&1&-\frac{\p_2\varphi}{\p_3\varphi}\\
0&0&\frac{1}{\p_3\varphi}
\end{pmatrix}^T,\quad\,\,
\A^{-1} =
\begin{pmatrix}
1&0&\p_1\vp\\
0&1&\p_2 \vp\\
0&0&\p_3\vp
\end{pmatrix}^T.
\end{equation*}

We prove \eqref{es-Hodge-1'} by induction. When $s=1$, we derive \eqref{es-Hodge-1'} from \eqref{Hodge-type} after squaring and integrating in space. We next assume $s>1$, and \eqref{es-Hodge-1'} holds for all $m\leq s-1$.  Let $\beta=(\beta_1, \beta_2, \beta_3)$ be a multi-index with $|\beta|=s-1$. We write
\begin{align}\label{Hodge higher}
|\p_i\p^\beta X| = |(\A^{-1})^j_i \p_j \p^\beta X|\leq C(|\psi|_{C^s})|\pp_i \p^\beta X|,
\end{align}
and then  invoke \eqref{Hodge-type} to arrive at
\begin{align}
|\p_i\p^\beta X|^2\leq C(|\psi|_{C^s})\left(|\pp\cdot (\p^\beta X)|^2+|\pp\times (\p^\beta X)|^2+|\TP \p^\beta X|^2\right).
\end{align}
This leads to
\begin{align}\label{Hodge pre}
\|X\|_s^2 \leq C(|\psi|_{C^s})\left(\|\pp\cdot (\p^\beta X)\|_0^2+\|\pp\times (\p^\beta X)\|_0^2+\|\TP \p^\beta X\|_0^2\right)
\end{align}
after integrating in space.
For the first term on the RHS of \eqref{Hodge pre}, we have
\begin{align*}
\|\pp\cdot (\p^\beta X)\|_0^2 \leq  \|\pp\cdot X\|_{s-1}^2+\|[\p^\beta, \pp\cdot] X\|_0^2\leq C(|\psi|_{C^s})\left(\|\pp\cdot X\|_{s-1}^2+\|X\|_s^2\right),
\end{align*}
where $\|X\|_s^2$ is covered by the inductive hypothesis. In addition, the second term on the RHS of \eqref{Hodge pre} is treated similarly. Finally, since $\TP$ commutes with $\p$, the last term in \eqref{Hodge pre} is just $\|\p^\beta (\TP X)\|_0^2$, which can be further reduced by repeating the steps above.
\end{proof}
\end{appendix}


\begin{thebibliography}{000}
\bibitem{Alazard2013GWP}
T.~Alazard and J.-M. Delort.
Global solutions and asymptotic behavior for two-dimensional gravity
  water waves.
\newblock {\it Ann. Sci. {\'E}c. Norm. Sup{\'e}r.,} {\bf 48} (2015), no.5, 1149--1238.

\bibitem{Alinhac-1989CPDE} S. Alinhac. Existence d'ondes de raréfaction pour des systèmes quasi-linéaires hyperboliques multidimensionnels. {\it Comm. Partial Differential Equations,} {\bf 14} (1989), no. 2, 173–-230.

\bibitem{BKM-1984CMP} J. T. Beale, T. Kato, A. Majda. Remarks on the breakdown of smooth solutions for the 3-D Euler equations. {\it Comm. Math. Phys.,} {\bf 94} (1984), no. 1, 61–66.

\bibitem{cheng2017solvability}
C.H. Cheng, S.~Shkoller.
\newblock Solvability and regularity for an elliptic system prescribing the
  curl, divergence, and partial trace of a vector field on Sobolev-class
  domains.
\newblock {\em Journal of Mathematical Fluid Mechanics}, {\bf 19} (2017), no. 3:375--422.

\bibitem{CCFGL}A. Castro, D. Córdoba, D. Fefferman, F. Gancedo, M. López-Fernández.  Rayleigh-Taylor breakdown for the Muskat problem with applications to water waves. Annals of Mathematics, 909--948, (2012).

\bibitem{CL2000priori}
D.~Christodoulou,  H.~Lindblad.
\newblock On the motion of the free surface of a liquid.
\newblock {\it Commun. Pure. Appl. Math.}, {\bf 53} (2000), no. 12, 1536--1602.

\bibitem{CCG} A. Córdoba, D. Córdoba,  F. Gancedo.  Interface evolution: water waves in 2-D. Advances in Mathematics, {\bf 223} (2010), no. 1, 120--173.


\bibitem{CS}
C. Coutand,  S. Shkoller.  Well-posedness of the free-surface incompressible Euler equations with or without surface tension. {\it Journal of the American Mathematical Society}, {\bf 20} (2007), no.3, 829--930.

\bibitem{CS2010}
C. Coutand, S. Shkoller. A simple proof of well-posedness for the free-surface incompressible Euler equations. {\it Discrete Contin. Dyn. Syst. Ser. S.}, {\bf 3} (2010), no. 3, 429--449.



\bibitem{Deng2017wwSTGWP}
Y.~Deng, A.D. Ionescu, B.~Pausader, and F.~Pusateri.
 Global solutions of the gravity-capillary water-wave system in three
  dimensions.
\newblock {\it Acta Math.}, {\bf 219}(2017), no.2, 213--402.

\bibitem{DK} M. M. Disconzi,  I. Kukavica. A priori estimates for the free-boundary Euler equations with surface tension in three dimensions. {\it Nonlinearity}, {\bf 32} (2019), no. 9, 3369.

\bibitem{DKT} M. M. Disconzi,  I. Kukavica., and A. Tuffaha. A Lagrangian interior regularity result for the incompressible free boundary Euler equation with surface tension.
{\it SIAM Journal on Mathematical Analysis}, {\bf 51}(2019), no. 5,3982--4022.

\bibitem{Ebin1987ill}
D.~G. Ebin.
The equations of motion of a perfect fluid with free boundary are not well-posed.
{\it Communications in Partial Differential Equations},
  {\bf 12} (1987), no. 10, 1175--1201.

\bibitem{Ferrari-1993CMP} A.B. Ferrari. On the blow-up of solutions of the 3-D Euler equations in a bounded domain. (English summary) {\it Comm. Math. Phys.} {\bf 155} (1993), no. 2, 277–294.

\bibitem{GuLuoZhang} X.~Gu, C.~Luo, and J.~Zhang. Zero surface tension limit of the free-boundary problem in incompressible magnetohydrodynamics. {\em Nonlinearity}. {\bf 35}, no. 12, p.6349.

\bibitem{GMS2012GWP}
P.~Germain, N.~Masmoudi, and J.~Shatah.
\newblock Global solutions for the gravity water waves equation in dimension 3.
\newblock {\it Ann. Math.} (2012), 691--754.

\bibitem{Dan}
D.~Ginsberg. On the breakdown of solutions to the incompressible Euler equations with free surface boundary.{\it SIAM Journal on Mathematical Analysis.}{\bf 53} (2011), no. 3, 3366--3384.

\bibitem{GLL}
D.~Ginsberg, H.~Lindblad, and C.~Luo.
\newblock Local well-posedness for the motion of a compressible,
  self-gravitating liquid with free surface boundary.
\newblock {\em Arch. Rational Mech. Anal.}, {\bf 236} (2020), no. 2:603--733.

\bibitem{HIT2016ww}
J.K. Hunter, M.~Ifrim, and D.~Tataru.
\newblock Two-dimensional water waves in holomorphic coordinates.
\newblock {\it Commun. Math. Phys.}, {\bf 346} (2016), no. 2, 483--552.

\bibitem{IT2016ww2}
M.~Ifrim and D.~Tataru.
\newblock Two-dimensional water waves in holomorphic coordinates ii: global
  solutions.
\newblock {\it Bulletin de la Soci\'et\'e math\'ematique de France}, {\bf 144} (2016), no. 2, 366--394.

\bibitem{IP}
A.D. Ionescu, and F. Pusateri. Global solutions for the gravity water waves system in 2D. {\it Inventiones mathematicae} {\bf 199} (2015), 653--804.

\bibitem{JLa}
V. Julin, and D. A. La Manna. A priori estimates for the motion of charged liquid drop: A dynamic approach via free boundary Euler equations. {\it arXiv:2111.10158v2 [math.AP]} (2021).

\bibitem{KTV}
I. Kukavica, A. Tuffaha, and V. Vicol. On the local existence and uniqueness for the 3D Euler equation with a free interface. {\it Applied Mathematics \& Optimization}, {\bf 76} (2017), 535--563.


\bibitem{KO}
I. Kukavica, and W. Ozanski. Local-in-time existence of a free-surface 3D Euler flow with $H^{2+\delta}$ initial vorticity in a neighborhood of the free boundary. {\it Nonlinearity}, {\bf 36} (2022), no.1,  636.


\bibitem{Lindblad2003LWP}
H.~Lindblad.
\newblock Well-posedness for the linearized motion of a compressible liquid
  with free surface boundary.
\newblock {\it Commun. Math. Phys.}, {\bf 236} (2003), no. 2, 281--310.

\bibitem{Luo-Zhang-2022} C. Luo, J. Zhang. Compressible Gravity-Capillary Water Waves with Vorticity: Local Well-Posedness, Incompressible, and Zero-Surface-Tension Limits. {\it arXiv:2211.03600v3 [math.AP]} (2022).

\bibitem{Nardi} G. Nardi. Schauder estimate for solutions of Poisson’s equation with Neumann boundary condition. L’enseignement Math\'ematique, 60(3), 421--435.

\bibitem{Shatah-Zeng-2008CPAM-1} J. Shatah, C. Zeng. Geometry and a priori estimates for free boundary problems of the Euler equation. {\it Comm. Pure Appl. Math.} {\bf 61} (2008), no. 5, 698–744.

\bibitem{Shatah-Zeng-2008CPAM-2} J. Shatah, C. Zeng. A priori estimates for fluid interface problems. {\it Comm. Pure Appl. Math.} {\bf 61} (2008), no. 6, 848–876.

\bibitem{Shatah-Zeng-2011ARMA} J. Shatah, C. Zeng. Local well-posedness for fluid interface problems. {\it Arch. Ration. Mech. Anal.} {\bf 199} (2011), no. 2, 653–705.

\bibitem{Taylor} M. Taylor. Partial Differential Equations III: Nonlinear Equations. Vol. 116. Springer Science \& Business Media, 2013.

\bibitem{WangZhang} C. Wang, Z. Zhang.  Breakdown criterion for the water-wave equation. {\it Science China Mathematics,} {\bf 60} (2017), 21--58.

\bibitem{WZZZ} C. Wang, Z. Zhang, W. Zhao, and Y. Zheng. Local well-posedness and break-down criterion of the incompressible Euler equations with free boundary. {\em Mem. Amer. Math. Soc.} {\bf 270} (2021), no. 1318, v+119 pp.

\bibitem{XW} X. Wang. Global infinite energy solutions for the 2D gravity water waves system. {\em Communications on Pure and Applied Mathematics}, {\bf 71}, no. 1, 90--162.

\bibitem{Wu1997LWP}
S.~Wu. Well-posedness in {S}obolev spaces of the full water wave problem in
  2-{D}.
{\it Invent. Math.} {\bf 130} (1997), no.1, 39--72.

\bibitem{Wu1999LWP}
S.~Wu.
 Well-posedness in {S}obolev spaces of the full water wave problem in
  3-{D}.
{\it J. Amer. Math. Soc.} {\bf 12} (1999), no.2, 445--495.

\bibitem{Wu2009GWP}
S.~Wu.
\newblock Almost global wellposedness of the 2-D full water wave problem.
\newblock {\it Invent. Math.} {\bf 177} (2009), no.1, 45--135.

\bibitem{Wu2011GWP}
S.~Wu.
\newblock Global wellposedness of the 3-D full water wave problem.
\newblock {\it Invent. Math.} {\bf 184} (2011), no. 1, 125--220.

\bibitem{ZhangZhang} P.~ Zhang,  Z.~ Zhang. On the free boundary problem of three‐dimensional incompressible Euler equations. {\it Comm. Pure Appl. Math.}, {\bf 61} (2008), no. 7, 877--940.
\end{thebibliography}
\end{document}